\documentclass[10pt,twocolumn,twoside]{IEEEtran}  

\IEEEoverridecommandlockouts                              

\usepackage{graphicx,subfigure} 
\usepackage{amsmath} 
\usepackage{amssymb}  
\usepackage{cite}
\usepackage{color}
\usepackage{bbm}
\usepackage{amsthm}

\newtheorem{theorem}{Theorem}[section]
\newtheorem{definition}[theorem]{Definition}

\newtheorem{lemma}[theorem]{Lemma}

\newtheorem{remark}[theorem]{Remark}

\newtheorem{proposition}[theorem]{Proposition}

\newcommand{\naturals}{\mathbb{N}}
\newcommand{\real}{\mathbb{R}}
\newcommand{\realnonneg}{\real_+}
\newcommand{\until}[1]{[#1]}
\newcommand{\map}[3]{#1:#2 \rightarrow #3}

\newcommand{\ones}{\mathbbm{1}}
\newcommand{\Lie}{\mathcal L}
\newcommand{\Pc}{\mathcal P}
\newcommand{\deriv}{{\operatorname{d}}}
\newcommand{\func}{{\operatorname{f}}}
\newcommand{\barr}{{\operatorname{p}}}
\newcommand{\dyn}{{\operatorname{dyn}}}
\newcommand{\des}{{\operatorname{des}}}
\newcommand{\diag}{{\operatorname{diag}}}
\newcommand{\off}{{\operatorname{off}}}

\newcommand{\longthmtitle}[1]{\mbox{}{\textit{(#1):}}}
\newcommand{\setdef}[2]{\{#1 \; | \; #2\}}
\newcommand{\setdefb}[2]{\big\{#1 \; | \; #2 \big\}}
\newcommand{\setdefB}[2]{\Big\{#1 \; | \; #2\Big\}}
\renewcommand{\SS}{\mathcal{S}}
\newcommand{\NN}{\mathcal{N}}
\newcommand{\Wc}{\mathcal{W}}
\newcommand{\pbb}{performance-barrier-based }
\newcommand{\PBB}{Performance-Barrier-Based }

\renewcommand{\theenumi}{(\roman{enumi})}

\renewcommand{\baselinestretch}{.99}

\allowdisplaybreaks

\title{\LARGE \bf \PBB Event-Triggered Control
  \\
  with Applications to Network Systems
  \thanks{A preliminary version of this
    paper appeared at the IEEE Conference on Decision and Control
    as~\cite{PO-JC:18-cdc}. This work was supported by NSF Award
    ECCS-1917177.}}

\author{Pio Ong and Jorge Cort\'es
  \thanks{P. Ong and J. Cort\'es are with Department of Mechanical and
    Aerospace Engineering, UC San Diego, {\tt\small
      \{piong,cortes\}@eng.ucsd.edu}}%
}

\begin{document}

\maketitle

\begin{abstract}
  This paper proposes a novel framework for resource-aware control
  design termed \pbb triggering. Given a feedback policy,
  along with a Lyapunov function certificate that guarantees its
  correctness, we examine the problem of designing its digital
  implementation through event-triggered control while ensuring a
  prescribed performance is met and triggers occur as sparingly as
  possible.  Our methodology takes into account the \emph{performance
    residual}, i.e., how well the system is doing in regards to the
  prescribed performance.  Inspired by the notion of control barrier
  function, the trigger design allows the certificate to deviate from
  monotonically decreasing, with leeway specified as an increasing
  function of the performance residual, resulting in greater
  flexibility in prescribing update times. We study different types of
  performance specifications, with particular attention to quantifying
  the benefits of the proposed approach in the exponential case. We
  build on this to design intrinsically Zeno-free distributed triggers
  for network systems.  A comparison of event-triggered approaches in
  a vehicle platooning problem shows how the proposed design meets the
  prescribed performance with a significantly lower number of
  controller updates.
\end{abstract}

\section{Introduction}

Trading computation and decision making for less actuator, sensing, or
communication effort offers great promises for the autonomous operation
of both individual and interconnected cyberphysical systems.  The
advent of increasingly capable devices operating in complex scenarios
raises the importance of using the available resources efficiently in
order to meet task specifications, prolong battery life, and provide
algorithmic solutions that can scale up.  Resource-aware control
examines the tight coupling between physical and cyber processes to
prescribe, in a principled way, when to use the available resources
while still guaranteeing a desired quality of service.  Motivated by
these observations, this paper develops an event-triggered control
framework that, given a prescribed performance specification,
incorporates in the decision making criteria the performance residual
to provide design flexibility for general nonlinear systems.

\subsubsection*{Literature Review} 
The event-triggered
framework~\cite{PT:07,WPMHH-KHJ-PT:12,LH-CF-HO-AS-EF-JPR-SIN:17} seeks
to determine criterions to employ opportunistically the available
control resources (e.g., actuation, sensing, communication) in order
to produce efficient implementations on digital systems. Such
criterions, called triggers, are commonly obtained by examining the
evolution under aperiodic sample-and-hold executions of the Lyapunov
certificates valid for their continuous-time counterparts. This can be
done in a derivative-based fashion, i.e., by monitoring the time
derivative of the certificate, see
e.g.,~\cite{PT:07,MA-RP-JD-DN:16,RP-PT-DN-AA:15,WPMHH-MCFD-ART:11,BAK-DJA-WPMHH:17},
or in a function-based fashion, i.e., by directly monitoring the value
of the certificate, see
e.g.,~\cite{MV-PM-EB:09,SD-NM-JFG:11,MMJ-AA-PT:09b}.  Both approaches
are widely applicable. However, derivative-based approaches tend to be
conservative because they are evaluated at the current system state
without taking into account how much the certificate has decreased
since the last update.  This is tackled in dynamic
event-triggering~\cite{AG:15} by introducing an extra variable to
store an estimate of this decrease and incorporate it into the trigger
design.  On the other hand, function-based designs suffer from lack of
robustness to disturbances in the value of the certificate.  The
work~\cite{AS-CP:11} uses both frameworks to mitigate these drawbacks
by estimating how much the certificate will decrease after each
trigger, which constitutes another source of conservatism, together
with its reliance on time triggering.  Here, we take a different
approach to combine the derivative- and function-based design
methodologies inspired by the concept of control barrier functions,
and particularly, Nagumo's Theorem, see
e.g.,~\cite{PW-FA:07,ADA-XX-JWG-PT:16,
  FB-SM:07,ADA-SC-ME-GN-KS-PT:19}.  The basic insight is to
incorporate into the trigger design the performance residual, i.e.,
how well the system is doing in regards to a prescribed performance
specification. This specification plays the role of the ``barrier''
that the system should not exceed.  This makes it possible to allow
the certificate to deviate from monotonically decreasing at all times,
with the amount of deviation allowed specified as a function of the
size of the performance residual.  Interestingly, the dynamic
event-triggered approach mentioned above can be naturally interpreted
within the framework proposed here.

Our technical approach also builds on the literature of
event-triggered approaches applied to the distributed control of
network systems, see
e.g.,~\cite{PW-MDL:09,MMJ-PT:10,PT-NC:14,DPB-VSD-WPMHH:16,DVD-EF-KHJ:12,CN-EG-JC:19-auto,JB-CN:21}
and references therein. One known issue in this context is that Zeno
behavior may arise as a result of the partial availability of
information to individual agents, despite it being ruled out for its
centralized counterpart. In such scenarios, it is common to use time
regularization~\cite{MMJ-PT:10,PT-NC:14,DPB-VSD-WPMHH:16}, i.e.,
preventing by design any update before certain fixed time (usually the
minimum inter-event time from the centralized design) has
elapsed. This requires an offline computation and the resulting
executions may behave like periodic time-triggered ones. An
alternative way of avoiding Zeno behavior is to allow for the
violation of the monotonic decrease of the certificate at all times,
see e.g.,~\cite{MCFD-WPMH:12,MG-DL-JSM-SD-KHJ:12}, at the cost of only
achieving practical stability.  Other works avoid Zeno behavior by
either requiring stronger system assumptions on the type of
certificates~\cite{XW-MDL:11,EG-PJA:12} or their solutions are
problem-specific~\cite{THC-ZK-JMS-WED:14,JB-CN:21}. Here, we combine
the performance-barrier-based framework with dynamic average
consensus~\cite{SSK-BVS-JC-RAF-KML-SM:19} to synthesize a Zeno-free
distributed design that ensure asymptotic convergence for a general
class of nonlinear systems.

\subsubsection*{Statement of Contributions}
This paper considers closed-loop continuous-time systems evolving
under a robustly stabilizing feedback endowed with a certificate in
the form of an ISS-Lyapunov function.  We address the problem of
developing a digital feedback implementation that simultaneously
retains the stability properties, opportunistically updates the
controller, and meets a prescribed performance.  The contributions of
the paper are threefold.  The first contribution is the synthesis of a
novel framework for event-triggered control termed
\pbb design.  We combine derivative- and
function-based designs by incorporating into the trigger criterion
both the time derivative and the value of the certificate.  The
flexibility of the proposed approach stems from allowing the
certificate to deviate from having to monotonically decrease at all
times. In our design, a larger performance residual, measured as the
difference between the prescribed performance and the value of the
certificate, results in a larger amount potential deviation
allowed. By construction, at any given state, the
performance-barrier-based design enjoys a longer inter-event time than
the derivative-based approach, while still achieving the prescribed
performance.  Our second contribution is the characterization of the
implementability and asymptotic stability properties of nonlinear
systems under the proposed framework.  We introduce the concept of
class-$\mathcal K$ performance specification function and establish,
for general nonlinear systems, a uniform lower bound in the
inter-event times of the proposed design, thereby ruling out the
possibility of Zeno behavior. For the particular case of exponential
performance specifications, which includes the case of linear control
systems, we provide an explicit expression of an improved minimum
inter-event time with respect to the derivative-based approach.  Our
third contribution builds on this characterization to develop
distributed triggers for network systems using the \pbb approach that
ensure asymptotic correctness.  Our distributed design makes use of
dynamic average consensus to estimate, with some tracking error, the
terms in the trigger criterion that require global information to be
evaluated.  The guarantees on the design then rely on its ability
to tolerate the tracking errors.  This is where we leverage the
flexibility provided by the \pbb approach to rule out Zeno behavior in
the network executions without using any time regularization.  We
conclude the paper by illustrating the effectiveness of the proposed
framework in a vehicle platooning problem.
 
\bigskip\bigskip

\section{Preliminaries}\label{sec:prelims}
This section presents basic preliminaries on graph theory
and dynamic average consensus\footnote{Throughout the paper, we use
  the following notation.  We denote by $\naturals$, $\real$ and
  $\realnonneg$, the set of natural, real and nonnegative real
  numbers, respectively. We let $\ones$ denote the vector with all its
  entries equal to one.  For $n \in \naturals$, we use $\until{n}$ to
  denote $\{1,\dots,n\}$. Given $x\in \real^n$, $\|x\|$ denotes its
  Euclidean norm. 
   We denote by
  $\mathbf{I} \in \real^{n \times n}$ the identity matrix.  A function
  $\map{f}{\real^n}{\real^n}$ is locally Lipschitz if, for every
  compact set $\SS_0\subset \mathbb R^n$, there exists $L>0$ such that
  $\|f(x)-f(y)\| \leq L \|x-y\|$, for all $ x,y \in \SS_0$. We use
  $\exp(\cdot)$ to denote the exponential function.  We let $\Lie_f$
  denote the Lie derivative along the vector field $f:\real^n
  \rightarrow \real^n$.  A continuous function $h:\mathbb R
  \rightarrow \mathbb R$ is of class-$\mathcal K$ if it is strictly
  increasing and $h(0)=0$. In addition, the function is
  class-$\mathcal K_\infty$ if it also satisfies $\lim_{r\rightarrow
    \infty} h(r)=\infty$.}.

\subsubsection*{Graph Theory}
Our exposition follows~\cite{FB-JC-SM:08cor}. We denote a graph by
$\mathcal G=(\mathcal V,\mathcal E)$, with $\mathcal V$ as the set of
vertices and $\mathcal E \subseteq \mathcal V \times \mathcal V$ as
the set of edges. We consider undirected graphs, where $(i,j)\in
\mathcal E$ implies $(j,i)\in \mathcal E$. A path between two vertices
$i,j\in \mathcal V$ is an ordered sequence of vertices starting with
$i$ and ending with $j$ such that all pairs of consecutive vertices
are elements of the set $\mathcal E$. A graph is connected if there
exists a path between any two vertices.  Vertices $i,j\in \mathcal V$
are neighbors if $(i,j)\in \mathcal E$. We let $\mathcal N_i$ denote
the set composed of vertex $i$ and all its neighbors. We add the
subscript $x_{\mathcal N_i}$ to represent the subvector of a vector
$x$ formed from the entries associated with~${\mathcal N_i}$. The
adjacency matrix $\mathbf A \in \real^{\vert \mathcal V \vert \times
  \vert \mathcal V \vert}$ has entries $\mathbf A_{ij}=\mathbf
A_{ji}=1$ if $i$ and $j$ are neighbors, and $\mathbf A_{ij}=\mathbf
A_{ji}=0$ otherwise. The degree of a node $i$ is $d(i) :=
\sum_{j\in\mathcal N_i} \mathbf A_{ij}$. The degree matrix $\mathbf D$
is the diagonal matrix with $\mathbf D_{ii} = d(i)$. The Laplacian
matrix $\mathbf L:= \mathbf D -\mathbf A$ has nonnegative real
eigenvalues and a simple eigenvalue of 0 with an eigenvector
$\ones$ iff the graph $\mathcal G$ is connected.

\subsubsection*{Dynamic Average Consensus}
Consider a group of $N$ agents communicating over an undirected
graph~$\mathcal G$.  Each agent $i\in \mathcal V = \until N$ has a
continuously differentiable reference signal
$\map{W_i}{[0,\infty)}{\real}$. Dynamic average consensus aims at
making the agents track asymptotically the average of the reference
signals. For convenience, let $W= (W_1,\dots,W_N)$. Here we employ
the dynamic average consensus
algorithm~\cite{SSK-BVS-JC-RAF-KML-SM:19},
\begin{equation}\label{sys:DAC}
  \dot y = \dot W -\rho\mathbf Ly ,
\end{equation}
where each component of $y\in \real^N$ is the agents' estimate of the
average, $\rho > 0$ is a rate of convergence parameter, and $\mathbf
L$ is the Laplacian matrix of the graph.  The following result shows
that with the correct initialization and a suitable assumption on the
evolution of~$W$, each state $y_i$ asymptotically tracks the average
$\ones^\top W(t)/N$ of the reference signal.  The result is a
refinement of~\cite[Thm.2]{SSK-BVS-JC-RAF-KML-SM:19} to reference
signals whose time derivative is bounded exponentially and its proof
is presented in the appendix.


\begin{lemma}\longthmtitle{Tracking Error Bound}\label{lem:trackbound}
  Consider the dynamic average consensus dynamics~\eqref{sys:DAC} with
  a reference signal $W$ whose time derivative is bounded
  exponentially, i.e., $\|\dot W(t)\|\leq c_{\dot W} \exp(-rt)$ with a
  constant $c_{\dot W}>0$, for time $t\in[0,s)$. Define the tracking
  error as $\epsilon := y-\ones\ones^\top W/N$. If the initialization
  of~$y$ is such that $\ones^\top y(0) =
  \ones^\top W(0)/N$, then the tracking error is also bounded for time
  $t\in[0,s)$ as
  \begin{multline}\label{eq:trackbound}
    \|\epsilon(t)\| \leq \frac{c_{\dot W}}{\rho\lambda_2-r} \exp(-rt)
    \\
    +\left(\|\epsilon(0)\| - \frac{c_{\dot
          W}}{\rho\lambda_2-r}\right)\exp(-\rho \lambda_2 t)
  \end{multline}
  where $\lambda_2$ is the second smallest eigenvalue of the Laplacian
  matrix~$\mathbf L$.~\hfill \qed
\end{lemma}

\section{Problem Formulation}\label{sec:problem}
%
Consider a nonlinear control system of the form
\begin{align*}
  \dot x = F(x,u),~x\in \real^n,~u\in \real^m ,
\end{align*}
with $\map{F}{\real^n \times \real^m}{\real^n}$.  The digital
implementation of a desired feedback policy
$\map{\kappa}{\real^n}{\real^m}$ as $u=\kappa(x)$ can be accomplished
through a sample-and-hold strategy.  This consists of updating the
control signal at a specific time $t_k$, for $k\in \{0\}\cup
\naturals$, and keeping it constant up until $t_{k+1}$, when the
evaluation of the feedback policy provides the next adjustment. As a
result, the closed-loop system~is
\begin{equation}\label{sys:NL}
  \dot x = F(x,\kappa(x+e))= f(x,e) ,
\end{equation}
where the error $e=x_k-x$ is the state deviation from the last update
at iteration~$k$ (here, we use the shorthand notation $x_k=x(t_k)$).
The challenge is then how to prescribe the sequence of update times
$\{t_k\}$ in order to ensure that the digital implementation retains
the convergence and performance properties of the original
continuous-time system.  

Event-triggered control looks past time-periodic implementations to
identify a state-dependent trigger criterion to determine the update
times.  To come up with such a criterion for a general nonlinear
system, a common starting point is to assume that there exists an
Input-to-State Stability (ISS) Lyapunov function for~\eqref{sys:NL},
see e.g.,~\cite{PT:07,MA-RP-JD-DN:16,RP-PT-DN-AA:15}. Formally, we
assume there exists a smooth function $V:\real^n \rightarrow \real$
and class-$\mathcal K_\infty$ functions $\underline \alpha$,
$\overline \alpha$, $\alpha$, and $\gamma$ satisfying
\begin{subequations}\label{eq:ISS}
  \begin{align}
    \underline \alpha(\|x\|) &\leq V(x) \leq \overline
    \alpha(\|x\|) , \label{eq:ISSLyap}
    \\
    \Lie_fV(x,e)&\leq -\alpha(\|x\|)+\gamma(\|e\|). \label{eq:ISSrate}
  \end{align}
\end{subequations}
The seminal work~\cite{PT:07} provides the trigger design
\begin{equation}\label{Trigger:Paulo}
  t_{k+1} = \setdefb{t \geq
    t_k}{-\sigma\alpha(\|x(t)\|)+\gamma(\|e(t)\|) = 0} ,
\end{equation}
with design parameter~$\sigma\in(0,1)$. Under~\eqref{Trigger:Paulo},
the rate of change of the Lyapunov function along~\eqref{sys:NL}
satisfies
\begin{align*}
  \frac{d}{dt} V(x(t)) \leq (\sigma-1)\alpha(\|x(t)\|).
\end{align*}
Therefore, by design, the certificate $V$ decreases along the
trajectories of the sample-and-hold implementation.  Stability cannot
be established from this fact alone, however, due to the possibility
of Zeno behavior: the state-dependency of the trigger criterion makes
it possible for the inter-event time between consecutive updates to become
increasingly small. This, in turn, leaves open the possibility of an
infinite number of updates within a finite period of
time. 
A common strategy to rule out Zeno behavior is to establish the
existence of a minimum inter-event time (MIET).  For the trigger
design~\eqref{Trigger:Paulo}, the existence of a MIET can be
established under mild assumptions, cf.~\cite{PT:07}.

\begin{figure}[tbh]
  \centering
  \subfigure[]{\fbox{\includegraphics[width=.425\linewidth]{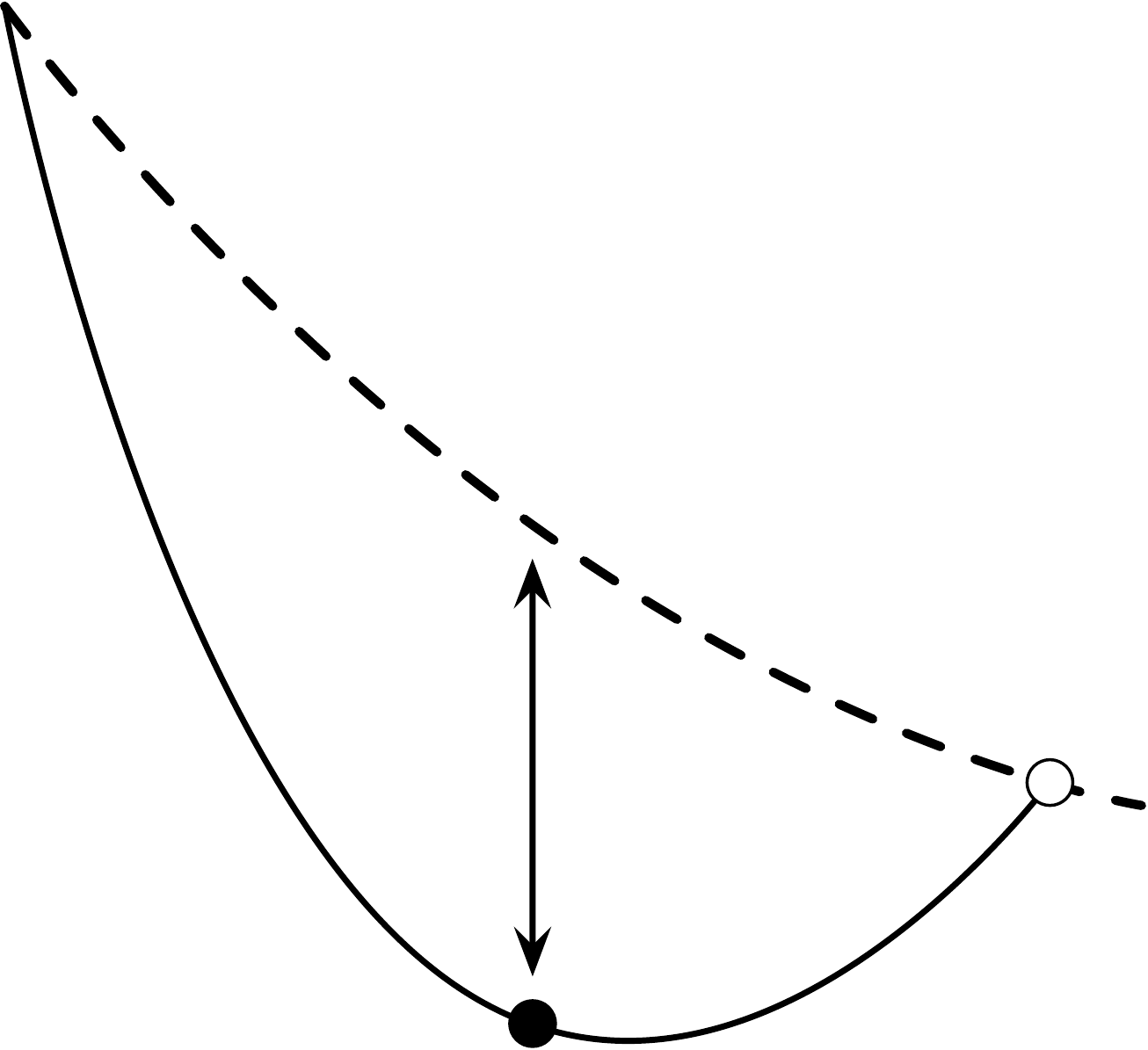}}}
  \;
  \subfigure[]{\fbox{\includegraphics[width=.425\linewidth]{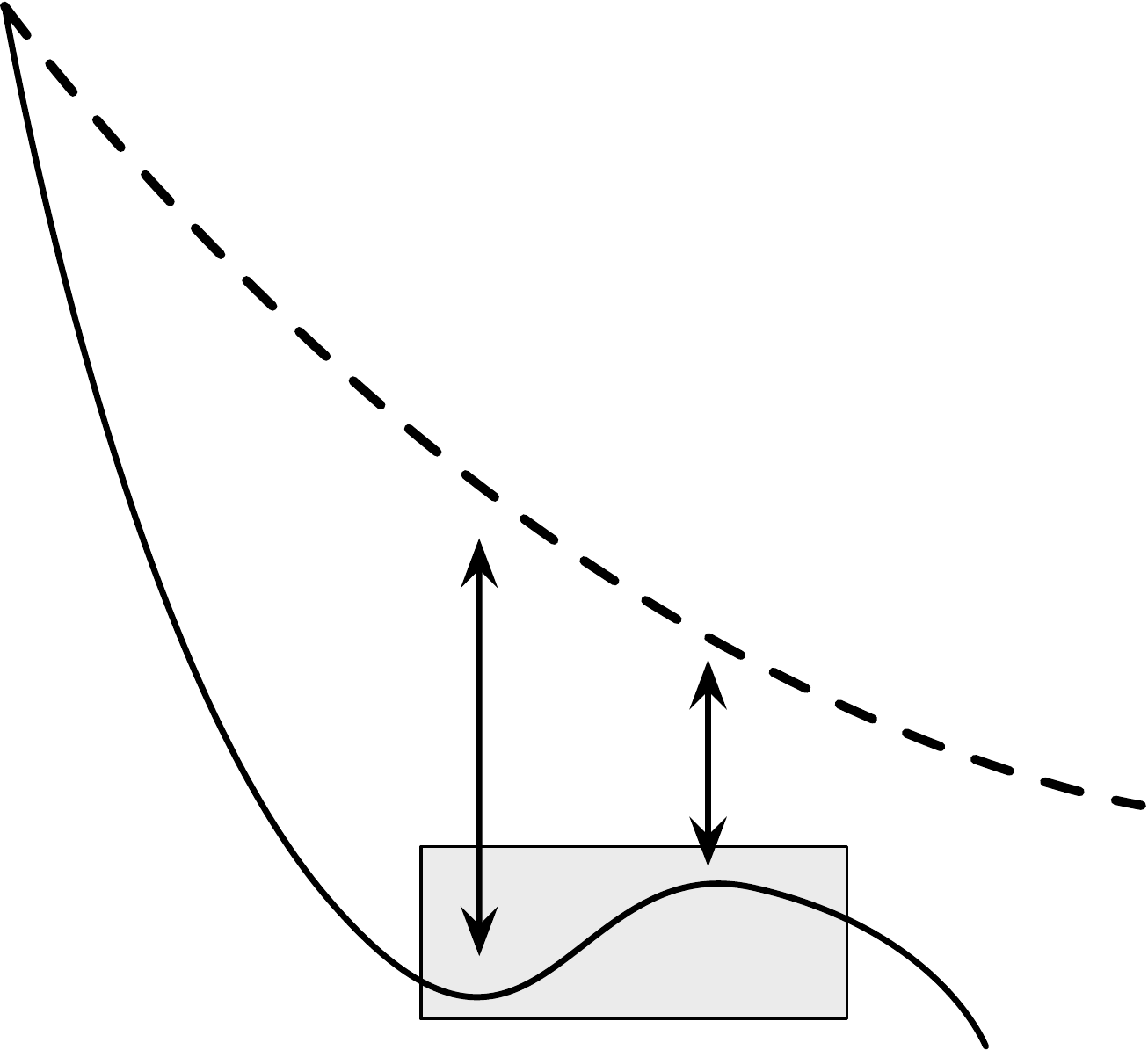}}}
  \caption{Prescribed performance (dashed line) and evolution of the
    certificate (solid line) under state-dependent triggering. (a) the
    controller update (black circle) prescribed
    by~\eqref{Trigger:Paulo} does not take into account the
    performance residual, which would otherwise be positive until the
    curve of the certificate meets the prescribed performance (empty
    circle). (b) a possible evolution of the certificate that
    momentarily violates (gray area) the derivative condition on the
    certificate specified by~\eqref{Trigger:Paulo}, does not require a
    controller update while always meeting the performance
    specification.}\label{fig:illustration-derivative}
\end{figure}

Triggering according to state-triggered criteria
like~\eqref{Trigger:Paulo} might lead to fewer controller updates than
a time-triggered implementation at the cost of impacting performance
(as measured, for instance, by the rate of decrease of the
certificate~$V$). Ideally, one would like the system to trigger as
sparingly as possible while still guaranteeing a prescribed
performance regarding convergence. In that
regard,~\eqref{Trigger:Paulo} tends to overprescribe updates, as the
criterion looks exclusively at the derivative of the certificate
without taking into account how much the certificate has decreased
since the last update,
cf.~Figure~\ref{fig:illustration-derivative}(a).  We refer to the
difference between the prescribed performance and the value of the
certificate as the \emph{performance residual}.  Presumably, allowing
the certificate to momentarily violate the derivative condition, with
leeway specified as an increasing function of the performance
residual, could result in executions with even fewer controller
updates that still meet the performance requirements,
cf. Figure~\ref{fig:illustration-derivative}(b). In the context of
network systems, the overprescription of controller updates is also
related to the fact that the design of distributed event-triggered
schemes based on~\eqref{Trigger:Paulo} might result, in general, in
sample-and-hold implementations that do not have a MIET,
see~\cite{WPMHH-KHJ-PT:12,MCFD-WPMH:12,PT-NC:14,DPB-WPMHH:14}.


The formalization of the ideas described above leads us to propose the
\emph{\pbb}design methodology for trigger design.  In
Section~\ref{sec:lin}, we limit our discussion to linear systems to
motivate and introduce the basic idea.  We develop it further for
general nonlinear systems in Section~\ref{sec:nonlin}. As we show in
our exposition, the new approach naturally leads to longer inter-event
times while meeting the specified performance. This provides the
necessary groundwork for tackling the design of Zeno-free
distributed event-triggered schemes for network systems in
Section~\ref{sec:dist}.

\section{\PBB Event-Triggered Control Designs for Linear
  Systems}\label{sec:lin}

Here we introduce the \pbb ETC framework. In this
section, we limit our discussion to linear systems for simplicity of
exposition.  Consider the sample-and-hold linear control system
\begin{equation}\label{sys:lin}
  \dot x =  Ax +BK x_k = (A+BK)x + BKe ,
\end{equation}
with matrices $A\in\real^{n\times n}$, $B\in \real^{n\times m}$ and
$K\in \real^{m \times n}$ so that $A+BK$ is Hurwitz. In this case, it
is easy to guarantee the existence of an ISS Lyapunov function
satisfying~\eqref{eq:ISS}. In fact, using the fact that $A+BK$ is
Hurwitz, there exists positive definite matrices $P$ and $Q$ such that 
\begin{subequations}\label{eq:Lyap-lin}
  \begin{equation}
    V(x) = x^\top P x
  \end{equation}
  is an ISS Lyapunov function with
  \begin{align}\label{eq:linrate}
    \Lie_fV(x,e) &=-x^\top Qx +2x^\top P BKe \notag
    \\
    &\leq \left(\frac{\|PBK\|}{\theta}- \lambda_{\min} (Q)\right)\|x\|^2
    + \theta\|PBK\|\|e\|^2 \notag
    \\
    &:= -c_\alpha \|x\|^2 + c_\gamma \|e\|^2 ,
  \end{align}
\end{subequations}
where $\lambda_{\min}(Q)$ is the minimum eigenvalue of~$Q$ and Young's
inequality~\cite{GHH-JEL-GP:52} is applied with $\theta>0$ selected
appropriately so that $c_\alpha$, $c_\gamma$ are positive. In
particular, for the original continuous-time system ($e\equiv 0$
in~\eqref{sys:lin}), one obtains the performance guarantee
\begin{align}\label{eq:LyapExpBound-cont}
  V(x(t)) \leq V(x_0) \exp \big ( c_\alpha \|P\|^{-1} t\big) ,
\end{align}
where $x_0$ denotes the initial condition. We next turn to the trigger
design.

\subsection{Derivative- and Function-Based Trigger Designs}

For the sample-and-hold linear system~\eqref{sys:lin}, the
derivative-based trigger design~\eqref{Trigger:Paulo} takes the form
\begin{align*}
  t_{k+1} = \min\setdefb{t\geq t_k}{-\sigma c_\alpha\|x\|^2+c_\gamma
    \|e\|^2 = 0} ,
\end{align*}
with the certificate along any trajectory satisfying $\frac{d}{dt}
V(x(t)) \leq (\sigma-1)c_\alpha\|x(t)\|^2$.  Using this inequality,
the evolution of the certificate satisfies
\begin{equation}\label{eq:LyapExpBound}
  V(x(t)) \leq V(x_0)\exp \big ( (\sigma-1) c_\alpha \|P\|^{-1} t\big)  .
\end{equation}
A higher value of $\sigma \in (0,1)$ results in a longer inter-event
time and a slower exponential rate on the evolution of the
certificate. This presents a trade-off for design. In order to compare
different designs fairly, it would seem reasonable to establish a
common performance criterion. Given the exponential convergence
characteristic of linear systems, prescribing a desired rate of
convergence $r>0$ is a natural candidate. Formally, we specify
\begin{equation}\label{eq:lin_req}
  V(x(t)) \leq V(x_0)\exp(-rt) ,
\end{equation}
at all time and for any initial condition. Given the
performance~\eqref{eq:LyapExpBound-cont} of the continuous
state-feedback system, we require $r< c_\alpha \|P\|^{-1}$.  Since the
derivative-based trigger is guaranteed to perform according
to~\eqref{eq:LyapExpBound}, one can see that $\sigma =
1-\frac{r\|P\|}{c_\alpha}$ is the value that yields the longest
inter-event time (for the derivative-based design) while still
satisfying the performance specification. The following result
summarizes the asymptotic convergence properties under the
derivative-based trigger design.



\begin{lemma}\longthmtitle{Derivative-Based Design -- Linear
    Case}\label{lem:deriv-lin} 
  Consider the sample-and-hold linear system~\eqref{sys:lin} with an
  ISS Lyapunov function~\eqref{eq:Lyap-lin}. Given a desired rate of
  convergence $r<c_\alpha\|P\|^{-1}$ and $\sigma \in
  (0,1-\frac{r\|P\|}{c_\alpha})$, let $\map{g}{\real^n \times
    \real^n}{\real}$ be any function such that
  \begin{align*}
    \Lie_fV(x,e)\leq g(x,e)\leq (\sigma-1)c_\alpha\|x\|^2+c_\gamma
    \|e\|^2.
  \end{align*}
  Define the derivative-based trigger time as
  \begin{equation}\label{Trigger:deriv-lin}
    t^\deriv_{k+1} = \min \setdefb{t\geq t_k}{g(x(t),e(t)) +rV(x(t))\geq 0}.
  \end{equation}
  There exists a MIET $\tau^\deriv_\sigma>0$ such that if
  $V(x(t_k))\leq V(x_0)\exp(-rt_k)$, then $ t^\deriv_{k+1}-t_{k} \ge
  \tau^\deriv_\sigma$.  As a consequence, if the trigger sequence
  $\{t_k\}_{k=0}^\infty$ is defined iteratively via the
  derivative-based trigger, then $V(x(t))< V(x_0)\exp(-rt)$ for all
  $t>0$, and the origin is globally exponentially
  stable. ~\hfill$\blacksquare$
\end{lemma}


Lemma~\ref{lem:deriv-lin} is essentially presented in~\cite{PT:07}. We
omit its proof as it is a special case of
Proposition~\ref{prop:barr-lin} below. The basic idea behind the
design~\eqref{Trigger:deriv-lin} is to keep the time derivative of the
Lyapunov function below an amount that, by application of the
Comparison Lemma~\cite[Lemma 3.4]{HKK:02}, would make the system
satisfy the desired performance, i.e., $\frac{d}{dt}V(x(t)) <
-rV(x(t))$. As a result, the gap $V(x_0)\exp(-rt)-V(x(t))$ between the
desired performance and the Lyapunov function, which we call
\emph{performance residual}, is always increasing until the next
update, see Figure~\ref{fig:illustration-derivative}. While meeting
the desired specifications means keeping the performance residual
nonnegative, doing so by having it always increase is overly
conservative.
To produce a less conservative design, one can instead look at the
value of the Lyapunov function itself (rather than its time
derivative), as specified in the following result.

\begin{lemma}\longthmtitle{Function-Based Design --
    Linear Case}\label{lem:func-lin}
  Consider the sample-and-hold linear system~\eqref{sys:lin} with an
  ISS Lyapunov function~\eqref{eq:Lyap-lin}. Given a desired rate of
  convergence $r<c_\alpha/\|P\|$, define
  the function-based trigger time as
  \begin{equation}\label{Trigger:func-lin}
    t^\func_{k+1} \!=\! \min\setdefb{t> t_k}{0\geq V(x_0)\exp(-rt) -V(x(t))}.
  \end{equation}
  There exists a MIET $\tau^\func_r>0$ such that if $V(x(t_k))\leq
  V(x_0)\exp(-rt_k)$, then $t^\func_{k+1}-t_{k} \geq \tau^\func_r $.
  As a consequence, if the trigger sequence $\{t_k\}_{k=0}^\infty$ is
  defined iteratively via the function-based trigger, then
  $V(x(t))\leq V(x_0)\exp(-rt_k)$, and the origin is globally
  exponentially stable.~\hfill$\blacksquare$
\end{lemma}

The function-based design relies on the idea of directly enforcing $V(x(t))\leq 
V(x_0)\exp(-rt)$. A problem with this design, however, is that it waits until 
the last moment, i.e., when the performance residual becomes zero (empty circle 
in Fig.~\ref{fig:illustration-derivative}(a)), to prescribe a controller update. 
Consequently, the implementation is not robust to errors (e.g., delays in 
evaluation or actual implementation).  The \pbb trigger design, 
proposed next, is motivated by the idea of overcoming the conservatism of the 
derivative-based design and the lack of robustness of the function-based one.

\subsection{\PBB Trigger Design}

Our ensuing design builds on the observation that to ensure that the
evolution of $V$ satisfies the specified performance, $V$ needs to
decrease faster than (or at the same rate as) the specification only
when their values are equal. Formally, this can be established using
Nagumo theorem~\cite {FB-SM:07}: $V(x(t))\leq V(x_0)\exp(-rt)$ if
and only if
\begin{align}\label{eq:condition}
  \frac{d}{dt} V(x(t))\leq -rV(x(t))~\text{when}~V(x(t))=
  V(x_0)\exp(-rt).
\end{align}
Note that this condition does not restrict how fast $V$ changes when
$V(x(t))< V(x_0)\exp(-rt)$, no matter how small the performance
residual is. One can readily see that the
condition~\eqref{eq:condition} suffers from the same lack of
robustness as the function-based design. To address this, and inspired
by how control barrier
functions~\cite{ADA-XX-JWG-PT:16,ADA-SC-ME-GN-KS-PT:19} restrict the
speed of their own evolution as the state approaches the boundary of
the safe set, we instead prescribe
\begin{align*}
  \frac{d}{dt} V(x(t))+rV(x(t))\leq
  c_\beta\big(V(x_0)\exp(-rt)-V(x(t))\big) ,
\end{align*}
with a nonnegative constant $c_\beta\geq 0$. The key idea is
restricting how fast $V$ can increase proportionally to the
performance residual.  The following result summarizes the asymptotic
convergence properties under this type of prescription.

\begin{proposition}\longthmtitle{\PBB Design --
    Linear Case}\label{prop:barr-lin}
  Consider the sample-and-hold linear system~\eqref{sys:lin} with an
  ISS Lyapunov function~\eqref{eq:Lyap-lin}. Given a desired rate of
  convergence $r<c_\alpha/\|P\|$ and $\sigma \in
  (0,1-\frac{r\|P\|}{c_\alpha})$, let $g$ be as in
  Lemma~\ref{lem:deriv-lin}. Define the \pbb trigger time
  as
  \begin{multline}\label{Trigger:barr-lin}
    t^\barr_{k+1} = \min \setdefb{t\geq t_k}{g(x(t),e(t)) +rV(x(t))
      \\\geq c_\beta\big(V(x_0)\exp(-rt) -V(x(t))\big)}.
  \end{multline}
  Let $ G(\tau) =\exp(A\tau)+\int_0^\tau \exp(A(\tau-s))ds BK$ and
  \begin{multline}
    M(\tau) =c_\beta P\exp(-r\tau)-c_\gamma\|\mathbf{I}-G(\tau)\|^2
    \\
    -G(\tau)^\top((c_\beta+r)P+(\sigma-1)c_\alpha \mathbf{I})G(\tau).
  \end{multline}
   The constant
  \begin{equation}\label{eq:inttime-lin}
    \tau^\barr_\sigma := \min\setdef{\tau>0}{\det(M(\tau))=0}, 
  \end{equation}
  is a MIET such that if $V(x(t_k))\leq V(x_0)\exp(-rt_k)$, then 
  $t^\barr_{k+1}-t_{k}\geq \tau^\barr_\sigma $.  As a consequence, if the trigger
  sequence $\{t_k\}_{k=0}^\infty$ is defined iteratively via the
  \pbb trigger, then $V(x(t))\leq V(x_0)\exp(-rt)$ for
  all time, and the origin is globally exponentially stable.
\end{proposition}
\begin{proof}
  First, we note that we can derive from the trigger design,
  $V(x(t))\leq V(x_0)\exp(-rt)$ for every interval
  $[t_k,t^\barr_{k+1})$, but we have omitted the proof here because it
  will appear in the proof of Proposition~\ref{prop:barr} later for
  the more general case. Nevertheless,
  we will prove here the result on the MIET, which will rule out the
  the sequence $\{t_k\}_{k=0}^\infty$ converging to a finite value
  (Zeno behavior).
  We start by deducing for each update
  \begin{align*}
    V(x(t_k))&\leq V(x_0)\exp(-rt_k)
    \\
    V(x(t_k))\exp(-r\Delta t_k) &\leq V(x_0)\exp(-rt)
  \end{align*}
  for the time $t \in [t_k,t_{k+1}]$ where $\Delta t_k = t-t_k$. Using
  this bound to lower bound the right-hand side of the trigger
  condition in~\eqref{Trigger:barr-lin}, as well as using the
  definition of $g$ to upper bound the left-hand side, we derive the
  condition
%
  \begin{multline}\label{eq:LinearTriggerCond}
    x^\top(rP+(\sigma-1)c_\alpha \mathbf{I})x +c_\gamma\|e\|^2
    \\
    = c_\beta(V(x_k)\exp(-r\Delta t_k)- x^\top Px)
  \end{multline}
  which must be met earlier. Note we have replaced inequality with
  equality due to continuity of all the terms along the
  trajectory. Under system \eqref{sys:lin}, we can find the expression
  for the state during each iteration as $x(t) =G(\Delta
  t_k)x_k$. Substituting the state and moving everything of the
  left-hand side to the right, \eqref{eq:LinearTriggerCond} becomes
  $$
  0=x_k^\top M(\Delta t_k)x_k.
  $$
  We know that $M(0)\succ 0$ because the right-hand side of
  \eqref{eq:LinearTriggerCond} is zero, and the left-hand side is
  negative at time $t_k$ due to the definition of~$r$. The MIET is
  given by when $M(\tau)$ transits from positive definite to
  semi-positive definite which is when there exists an $x_k$ such that
  the condition is satisfied. Therefore, the MIET is given
  by~\eqref{eq:inttime-lin}. As a result, $V(x(t))< V(x_0)\exp(-rt)$
  for all time. Lastly, the origin can be deemed exponentially stable
  as we can derive
  \begin{align*}
    \|x\| \leq
    \|x_0\|\frac{\|P\|^{1/2}}{\lambda_{\min}(P)^{1/2}}\exp(-rt/2),
  \end{align*}
  concluding  the proof.
\end{proof}


Proposition~\ref{prop:barr-lin} generalizes both
Lemmas~\ref{lem:deriv-lin} and~\ref{lem:func-lin}. Note that the
trigger design~\eqref{Trigger:deriv-lin} is recovered by selecting
$c_\beta = 0$ in~\eqref{Trigger:barr-lin}, and the trigger
design~\eqref{Trigger:func-lin} corresponds to the limit
of~\eqref{Trigger:barr-lin} as $c_\beta \rightarrow \infty$. Directly
from the construction of the trigger designs, one can deduce
$t^\deriv_{k+1}\leq t^\barr_{k+1}\leq t^\func_{k+1}$ (inequalities are
strict if $g$ is continuous). Therefore, we can adjust the parameter
$c_\beta$ to control the inter-event times, which is also evident in
the expression for the MIET. Note that the \pbb design enjoys longer
inter-event times than the derivative-based one while still being able
to achieve the prescribed performance. Although the \pbb strategy does
not have a MIET as large as the function-based one, it does not suffer
from the same lack of robustness to errors
The design also includes the flexibility of using the surrogate
function~$g$ if it is more convenient or easier to evaluate.  Finally,
Proposition~\ref{prop:barr-lin} also provides a method to calculating
the MIET using the design~\eqref{Trigger:barr-lin} for linear control
systems.  The expression only depends on time (not on the state),
which means that it can be calculated offline.


\section{\PBB Event-Triggered Control Designs for
  Nonlinear Systems}\label{sec:nonlin}

In this section we expand our presentation of the \pbb
event-triggered control design to general nonlinear
systems~\eqref{sys:NL}.  Our starting point is the availability of an
ISS Lyapunov function~\eqref{eq:ISS} in tandem with the feedback
policy~$\kappa$. Unlike the case of linear systems, the evolution of
the Lyapunov function along the trajectories of the closed-loop system
might not be exponentially decaying, and this raises the question of
how to suitably define a performance specification. We do this by
considering a continuously differentiable, time-dependent function
$\map{S(\cdot;x_0)}{\real_+}{\real_+}$, parametrized by the initial
condition~$x_0$,  encoding the desired behavior~as
\begin{equation}\label{eq:requirement}
  V(x(t)) \leq S(t;x_0).
\end{equation}
We use Nagumo theorem~\cite{FB-SM:07} to write an equivalent condition
(assuming that $V(x_0)\leq S(0;x_0)$) to the
requirement~\eqref{eq:requirement} as
\begin{equation}\label{eq:nagumo}
  \frac{d}{dt}V(x(t)) \leq \frac{d}{d t}S(t;x_0) \text{ when } V(x(t))
  = S(t;x_0).
\end{equation}
With this in mind, we seek to identify different types of
performance specification functions $S$ that allow us to establish the
existence of a MIET. In the following, we discuss several classes of
specification functions.

\subsection{Class-$\mathcal K$ Derivative Performance
  Specification}\label{sec:classk}
This class of specification function is an extension of the
exponential decrease of the linear case. 
In particular, note that the desired convergence rate $r$ is limited
in the linear case by the performance~\eqref{eq:LyapExpBound-cont} of
the original continuous-time system. Similarly, in the nonlinear case,
we look at the performance under the continuous-time controller
implementation ($e \equiv 0$ in~\eqref{sys:NL}). Hence, let
$\map{h}{\real_+}{\real_+}$ be such that
\begin{align*}
  \Lie_fV(x,0)\leq -\alpha(\|x\|)<-h(V(x)),
\end{align*}
for all~$x$. In other words, $h$ expects a slower convergence than the natural
convergence of the system with a continuous controller. 

\begin{definition}\longthmtitle{Class-$\mathcal K$ Derivative
    Specification}\label{def:classk}
  {\rm For $\sigma^*\in(0,1)$, let $\map{h}{\real_+}{\real_+}$ be
    locally Lipschitz and class-$\mathcal K$ with $h(V(x)) \leq
    (1-\sigma^*) \alpha(\|x\|)$ for all $x$.  A function
    $\map{S(\cdot;x_0)}{\real_+}{\real_+}$ is a class-$\mathcal K$
    derivative performance specification if it is the unique solution
    to the differential equation
    \begin{align*}
      \dot S = -h(S), \quad S(0;x_0) \geq V(x_0),
    \end{align*}
    for any initial condition~$x_0$.~\hfill$\bullet$ }
\end{definition}

According to this definition, $S$ is strictly decreasing in time and
$\lim_{t\rightarrow\infty}S(t;x_0)=0$ for all $x_0$, and is increasing
in $\|x_0\|$, cf. \cite[Lemma 4.4]{HKK:02} (with a slight abuse of
notation, writing the specification in the form $S(\|x_0\|,t)$ makes
it a class $\mathcal{KL}$ function). Note that the exponential rate
specification is a particular case of Definition~\ref{def:classk} (by
setting $h(s) = -r \, s$).  The following result expands the treatment
in~\cite{PT:07} regarding derivative-based triggers to account for
this notion of performance specification and follows a similar line of
reasoning.

\begin{proposition}\longthmtitle{Derivative-Based Design --
    Class-$\mathcal K$ Derivative}\label{prop:deriv}
  Consider the sample-and-hold nonlinear system~\eqref{sys:NL} with an
  ISS Lyapunov function~\eqref{eq:ISS}. Given a class-$\mathcal K$
  derivative performance specification~$S$ and $\sigma \in
  (0,\sigma^*)$, let $\map{g}{\real^n \times \real^n}{\real}$ be any
  function such that
  \begin{align*}
    \Lie_fV(x,e)\leq g(x,e)\leq (\sigma-1)\alpha(\|x\|)+\gamma(\|e\|).
  \end{align*}
  Define the derivative-based trigger time as
  \begin{equation}\label{Trigger:deriv}
    t^\deriv_{k+1} =\min \setdefb{t \geq t_k}{g(x(t),e(t))+ h(V(x(t)))
      \geq0}.
  \end{equation}
  Under the assumption that $F$, $\kappa$, $\gamma$, $\alpha^{-1} $
  are locally Lipschitz, there exists a MIET $\tau^\deriv_\sigma > 0$
  such that if $V(x(t_k))\leq S(t_k;x_0)$, then $t^\deriv_{k+1}-t_k
  \geq \tau^\deriv_\sigma$.  As a consequence, if the sequence
  $\{t_k\}_{k=0}^\infty$ is defined iteratively via the
  derivative-based trigger, then $V(x(t))\leq S(t;x_0)$ for all time,
  and the origin is globally asymptotically stable.
\end{proposition}
\begin{proof}
  The trigger design directly enforces $ g(x(t),e(t))< -h(V(x(t))$ for
  $t \in [t_k,t^\deriv_{k+1})$. Therefore, 
  \begin{align*}
    \frac{d}{dt}V(x(t)) = \Lie_f V(x(t),e(t)) \leq -h(V(x(t))).
  \end{align*}
  Consequently, if $ V(x(t_k))\leq S(t;x_0)$, one can guarantee
  $V(x(t)) \leq S(t;x_0)$ for all $t\in[t_k,t^\deriv_{k+1}]$ via the
  Comparison Lemma~\cite[Lemma 3.4]{HKK:02}.  Next, we prove the
  existence of a MIET. Because the sublevel set $\setdef{x \in
    \real^n}{V(x)\leq S(0;x_0)}$) is forward invariant and compact,
  $\|e\|=\|x-x_k\|$ must be bounded by some constant $E>0$ and hence
  the error remains in the compact set $\setdef{e}{\|e\|\leq E}$. On
  these compact sets, let $L_\gamma$ and $L_{\alpha^{-1}}$ denote the
  Lipschitz constants for the functions $\gamma$ and $\alpha^{-1} $,
  respectively. Then,
  \begin{align*}
    t^\deriv_{k+1} &= \min \setdefb{t \ge t_k}{g(x,e)+h(V(x))\geq0}
    \\
    &\geq  \min \setdefb{t \ge
      t_k}{g(x,e)+(1-\sigma^*)\alpha(\|x\|)\geq0}
    \\
        &\geq \min \setdefb{t \ge
      t_k}{(\sigma-\sigma^*)\alpha(\|x\|)+\gamma(\|e\|)=0}
    \\
    &\geq \min \setdefb{t \ge
      t_k}{\frac{L_\gamma}{\sigma^*-\sigma}\|e\|=\alpha(\|x\|)}
    \\
    &\geq \min \setdefb{t \ge
      t_k}{\frac{L_{\alpha^{-1}}L_\gamma}{\sigma^*-\sigma}\|e\|=\|x\|}
    \\
    &=\min \setdefb{t \ge t_k}{\frac{\|e\|}{\|x\|}= D^{-1}} ,
  \end{align*}
  where $D = \frac{L_{\alpha^{-1}}L_\gamma}{\sigma^*-\sigma}$. 
  Using Lemma~\ref{lem:errBnd}, the time at which the condition in the
  last equation is met is lower bounded by $t_k + \frac{1}{L_fD+L_f}$,
  where $L_f$ is the Lipschitz constant for $f$ with respect to
  $(x,e)$ (which exists because $F$ and $\kappa$ are locally
  Lipschitz). This establishes the existence of a positive MIET bound,
  ruling out the possibility of Zeno behavior in the sequence
  $\{t_k\}_{k=0}^\infty$. Finally, asymptotic stability follows from
  the fact that $S$ is strictly decreasing and $\lim_{t\rightarrow
    \infty}S(t;x_0) = 0$, concluding the proof.
\end{proof}

Next, we build on the ideas presented in Sections~\ref{sec:problem}
and~\ref{sec:lin} to introduce the \pbb trigger
design~\eqref{Trigger:barr} for the nonlinear case.
The proposed design is based on enforcing 
the condition~\eqref{eq:nagumo} to ensure the performance
specification is met. In doing so, we take advantage of the
\emph{performance residual} $S(t;x_0)-V(x(t))$ to avoid
overconstraining the evolution of the Lyapunov certificate $V$ when
$V(x(t))<S(t;x_0)$.

\begin{proposition}\longthmtitle{\PBB
    Design -- Class-$\mathcal K$ Derivative}\label{prop:barr}
  Consider the sample-and-hold nonlinear system~\eqref{sys:NL} with an
  ISS Lyapunov function~\eqref{eq:ISS}. Given a class-$\mathcal K$
  derivative performance specification $S$ and $\sigma \in
  (0,\sigma^*)$, let $g$ be as in Proposition~\ref{prop:deriv} and let
  $\beta$ be any $\mathcal K_\infty$ function on $[0,\infty)$. Define
  the \pbb trigger time as
  \begin{multline}
    t^\barr_{k+1} =\min \setdefb{t \geq t_k}{g(x(t),e(t))+ h( V(x(t)))
      \\
      \geq \beta\big(S(t;x_0)- V(x(t))\big)}.\label{Trigger:barr}
  \end{multline}
  Under the assumption that $F$, $\kappa$, $\gamma$, $\alpha^{-1}$ are
  locally Lipschitz, there exists a MIET $\tau^\barr_\sigma>0$ such
  that if $V(x(t_k))\leq S(t;x_0)$,
  then $t^\barr_{k+1}-t_k \geq \tau^\barr_\sigma$. As a consequence,
  if the sequence $\{t_k\}_{k=0}^\infty$ is defined iteratively via
  the \pbb trigger, then $V(x(t))\leq S(t_k;x_0)$ for
  all time, and the origin is globally asymptotically stable.
\end{proposition}
\begin{proof}
  The trigger design directly enforces
  \begin{equation}\label{eq:perf_barr}
    g(x(t),e(t)) +h(V(x(t))) < \beta\big(S(t;x_0) - V(x(t))\big) ,
  \end{equation}
  for $t \in [t_k,t^\barr_{k+1})$. Thus, when $S(t;x_0) = V(x(t))$, we
  find $ g(x(t),e(t)) +h(S(t;x_0))< 0 $, and hence $
  \Lie_fV(x(t),e(t)) < -h(S(t;x_0))$ from the properties of~$g$. Since
  $S$ is a derivative performance specification, it follows that
  \begin{align*}
    \frac{d}{dt}V(x(t)) &< \frac{d}{d t}S(t;x_0),
  \end{align*}
  implying~\eqref{eq:nagumo}.  Consequently, $V(x(t)) \leq S(t;x_0)$
  for all $t\in[t_k,t^\barr_{k+1})$.
  We can also use this fact to deduce
  \begin{align*}
    t^\barr_{k+1} &\geq \min \setdefb{t \ge t_k}{g(x(t),e(t))+h(V(x(t)))=0},
  \end{align*}
  implying that the \pbb trigger time must occur after
  the derivative-based trigger one~\eqref{Trigger:deriv}. Thus,
  $\tau^\deriv_\sigma$ from Proposition~\ref{prop:deriv} is a valid MIET for
  the \pbb trigger as well, ruling out the possibility
  of Zeno behavior in $\{t_k\}_{k=0}^\infty$. Finally, asymptotic
  stability follows from the properties of~$S$.
\end{proof}

Note that the function~$\beta$ in Proposition~\ref{prop:barr}
restricts the speed of evolution of the Lyapunov certificate~$V$ when
$V(x(t))<S(t;x_0)$ as a function of the performance residual.

\begin{remark}\longthmtitle{Comparison with Derivative-Based
    Approach: Longer Inter-Event Times}\label{remark:IncIntTime} {\rm
    As pointed out by Propositions~\ref{prop:deriv}
    and~\ref{prop:barr}, both the derivative- and \pbb
    approaches meet the performance specification defined
    by~$S$. However, since the performance residual on the right-hand
    side of~\eqref{Trigger:barr} always remains greater than zero by
    design, the \pbb approach, for a given system state,
    has a longer inter-event time than the derivative-based one, and
    is therefore less conservative.  In general, it is challenging to
    provide an explicit bound between the respective MIETs due to the
    generality of the system dynamics and the performance
    requirement. We show later in Section~\ref{sec:exp} that in the
    case of exponential performance specification this difference in
    MIETs can be quantified analytically.~\hfill $\bullet$}
\end{remark}

\begin{remark}\longthmtitle{Comparison with Function-Based Approach:
    Robustness to Input 
    Disturbances}\label{remark:Robustness}
  {\rm 
    A purely function-based design would correspond
    to~\eqref{Trigger:barr} with the left-hand side substituted by
    zero. Note that the error term does not show up explicitly in such
    design, in contrast to the \pbb approach.  Much like
    how one can use the ISS notion to deal with disturbances, the
    \pbb design allows for the analysis and
    mitigation of input disturbances.
    This is the intention of the presence of the parameter~$\sigma$ in
    the definition of $g$, that reserves a part of the negativity of
    the Lyapunov function decay.~\hfill $\bullet$}
\end{remark}


\begin{remark}\longthmtitle{Connection with Dynamic Trigger
    Design}\label{re:dynamic-trigger}
  {\rm We note that dynamic triggering can be interpreted as a
    particular case of the \pbb trigger design, where the
    performance function is specified in an online fashion. We
    elaborate on this point here.  Formally, and with the same
    notation employed in Proposition~\ref{prop:barr}, the dynamic
    trigger~\cite{AG:15} would take the form
    \begin{subequations}\label{eq:dynamic}
      \begin{align}
        t^\dyn_{k+1} & = \min\setdefb{t\geq t_k}{\theta g(x(t),e(t))
          \geq \eta(t)}, \label{Trigger:dynamic}
        \end{align}
        for~$\theta>0$, where the variable $\eta$ follows the dynamics
        \begin{align}
        \dot \eta & = -\iota(\eta) -
        g(x,e) \label{eq:dynamic_storage}
      \end{align}
    \end{subequations}
    with a locally Lipschitz class-$\mathcal K_\infty$
    function~$\iota$.  The basic idea is to store the decrease of $V$
    in the variable $\eta$ through~\eqref{eq:dynamic_storage} and use
    it to increase the inter-event times
    in~\eqref{Trigger:dynamic}. The term $\iota(\eta)$ represents a
    decay in the stored amount, ensuring that the system as a whole
    loses total ``energy'' over time.
    
    Interestingly, the dynamic design~\eqref{eq:dynamic} can be
    interpreted from the perspective of \pbb ETC.
    Selecting the performance specification function $S (t;x_0) = \eta
    (t) + V (x(t; x_0))$, one can see that the
    design~\eqref{eq:dynamic} ensures
    \begin{align*}
      \frac{d}{dt} V(x(t;x_0)) - \frac{d}{dt}S(t;x_0) <
      \beta(S(t;x_0)-V(x(t; x_0))
    \end{align*}
    with $\beta(\eta)=\iota(\eta)+\eta /\theta$ (note the parallelism
    with the \pbb design~\eqref{Trigger:barr}),
    implying~\eqref{eq:nagumo} is satisfied.  Note that this
    performance specification~$S$ is not known a priori and is instead
    determined in an online fashion, tailored to the concrete initial
    condition of the system trajectory. In particular, this means that
    the explicit performance guarantee of the design is difficult to
    obtain unless additional assumptions are made on the dynamics.  A
    final observation is that errors in the evaluation of the decrease
    of $V$ might jeopardize the convergence properties of dynamic
    triggering, whereas the evaluation of the performance residual in
    a feedback fashion characteristic of the \pbb ETC
    approach makes it naturally robust to errors.~\hfill $\bullet$ }
\end{remark}

\subsection{Exponential Performance Specification}\label{sec:exp}

Here we discuss the exponential performance specification. This is a
subfamily of the class-$\mathcal K$ derivative performance
specifications in Section~\ref{sec:classk} for which an explicit
analysis of the performance residual leads us to an improved MIET with
respect to the derivative-based approach.

In this case, in lieu of the conditions~\eqref{eq:ISS} for the ISS
Lyapunov function $\map{V}{\real^n}{\real}$, assume the following
stronger set of conditions hold: there exist positive constants $c_1$,
$c_2$, $c_3$, and $c_4$ such that
\begin{subequations}\label{eq:ExpLyap}
  \begin{align}
    c_1\|x\|^2 \leq V(x) & \leq c_2\|x\|^2, \label{eq:ExpLyapA}
    \\
    \frac{dV}{dx} f(x,0) & \leq -c_3\|x\|^2, \label{eq:ExpLyapB}
    \\
    \left\|\frac{dV}{dx}\right\| & \leq c_4\|x\| , \label{eq:ExpLyapC}
  \end{align}
\end{subequations}
for all $x \in \real^n$. Under the additional assumption that $F$ and
$\kappa$ are globally Lipschitz, and using Young's
inequality~\cite{GHH-JEL-GP:52}, the following inequality holds for
all $(x,e)$,
\begin{align}
  \Lie_fV(x,e) = \frac{dV}{dx} f(x,e) &\leq
  -c_3\|x\|^2+c_4L_f\|x\|\|e\| \notag
  \\
  &\leq -c_\alpha\|x\|^2+c_\gamma\|e\|^2  , \label{eq:expISS}
\end{align}
for some positive constants $L_f,c_\alpha$, and $c_\gamma$. Notice
that the functions $\underline \alpha$, $\overline \alpha$, $\alpha$,
and $\gamma$ for this ISS Lyapunov function are defined as quadratic
functions with constants $c_1$, $c_2$, $c_\alpha$, $c_\gamma$,
respectively. Note that in the absence of error, the value of $V$
converges exponentially. Hence, we consider the exponential
performance specification $S(t;x_0) = V(x_0)\exp(-rt)$ with
$r<c_\alpha/c_2$, which is of class-$\mathcal K$ since it is the
unique solution to $\dot S = -rS$, cf. Definition~\ref{def:classk}.

The next result provides an expression for the MIET for the
\pbb trigger design~\eqref{Trigger:barr} and shows it is
strictly larger than the MIET $\tau^\deriv_\sigma$ of the
derivative-based trigger design.

\begin{proposition}\longthmtitle{\PBB Design --
    Exponential Performance}\label{prop:exp}
  Consider the sample-and-hold nonlinear system~\eqref{sys:NL} with a
  Lyapunov function~\eqref{eq:ExpLyap}. Given an exponential
  performance specification $S$ and $\sigma \in
  (0,1-\frac{rc_2}{c_\alpha})$, let $g$ be as in
  Proposition~\ref{prop:deriv}, and $\beta(z) = c_\beta z$ with a
  positive $c_\beta$.
  Define
  \begin{multline}\label{eq:tau-exp}
    \tau^{\exp}_\sigma := \min \setdefB{\tau\geq
      0}{(\xi(\tau)+r)\exp\left(\int_{0}^\tau
        \xi(s)ds\right)
      \\
      = c_\beta \left(\exp (-r\tau)-\exp\left(\int_{0}^\tau
          \xi(s)ds\right)\right)}
  \end{multline}
  where
  \begin{align*}
    \tau^\deriv_\sigma &= \xi^{-1}(-r):= \frac{\sqrt{((1-\sigma)c_\alpha-r)
        /c_\gamma}}{L_f+L_f\sqrt{((1-\sigma)c_\alpha-r)/c_\gamma}},
    \\
    \xi(\tau) &= \begin{cases}
      ((\sigma-1)c_\alpha+c_\gamma\phi(\tau)^2)/c_2 & 0\leq \tau<
      \tau^*
      \\
      ((\sigma-1)c_\alpha+c_\gamma\phi(\tau)^2)/c_1 & \tau^*\leq \tau
    \end{cases},
    \\
    \phi(\tau) &= \frac{L_f \tau}{1-L_f \tau}, ~\tau^* = \xi^{-1}(0)
    :=\frac{\sqrt{(1-\sigma)c_\alpha
        /c_\gamma}}{L_f+L_f\sqrt{(1-\sigma)c_\alpha/c_\gamma}} .
  \end{align*}
  Under the assumption that $F$ and $\kappa$ are globally Lipschitz,
  $\tau^{\exp}_\sigma$ is a MIET such that if $V(x(t_k))\leq
  V(x_0)\exp(-rt_k)$, then $t^{\barr}_{k+1}-t_k \geq
  \tau^{\exp}_\sigma>\tau^\deriv_\sigma$.  As a consequence, if the
  trigger sequence $\{t_k\}_{k=0}^\infty$ is defined iteratively with
  the exponential \pbb trigger~\eqref{Trigger:barr}, then
  $V(x(t))\leq V(x_0)\exp(-rt)$ for all time, and the origin is
  globally exponentially stable.
\end{proposition}
\begin{proof}
  The statements on performance and stability follow with the same
  arguments used in the proof of Proposition~\ref{prop:barr}. Here, we
  only establish the MIET expression given for
  $\tau^{\exp}_\sigma$. First, we use \eqref{eq:ExpLyapA} and
  Lemma~\ref{lem:errBnd} to find
  \begin{align*}
    g(x(t),e(t)) &\leq
    ((\sigma-1)c_\alpha+c_\gamma\phi(t-t_k)^2)\|x\|^2
    \\
    &\leq \xi(t-t_k)V(x) ,
  \end{align*}
  where we have used that $\phi(\tau^*)^2 =
  (1-\sigma)c_\alpha/c_\gamma$. This gives the bound
  \begin{multline*}
    t^{\barr}_{k+1} \geq \min \setdefb{t\geq t_k}{(\xi(t-t_k)+r)V(x(t))
      \\\geq c_\beta(V(x_0)\exp (-rt)-V(x(t)))}.
  \end{multline*}
  In addition, we can bound the Lyapunov function along the trajectory
  using the differential form of Gronwall's inequality~\cite[Lemma
  A.1]{HKK:02} as
  \begin{equation}\label{eq:gronLyap}
    V(x(t)) \leq V(x_k) \exp \Big(\int_{t_k}^t \xi(s-t_k)ds \Big).
  \end{equation}
  This helps us isolate the state component, which in turn allows us
  to bound the trigger time with only the time variable as follows
    \begin{multline*}
    t^{\barr}_{k+1} \geq \min \setdefB{t\geq
      t_k}{(\xi(t-t_k)+r)\exp\left(\int_{t_k}^t \xi(s)ds\right)
      \\ \geq c_\beta\left(\exp (-r(t-t_k))-\exp\left(\int_{t_k}^t
          \xi(s-t_k)ds\right)\right)}.
  \end{multline*}
  With the change of variables $\tau = t-t_k$, and using continuity,
  the condition defining the set is as in~\eqref{eq:tau-exp}. Next,
  because $\xi$ is strictly increasing and $\xi(\tau^\deriv_\sigma) =
  -r$, the left-hand side of the condition is nonpositive for $\tau
  \leq \tau^\deriv_\sigma$. At the same time, the right-hand side of
  the condition must always be positive. Hence, the condition must be
  met at $\tau^{\exp}_\sigma> \tau^\deriv_\sigma$, concluding the
  proof.
\end{proof}

Note that the expression~\eqref{eq:tau-exp} in
Proposition~\ref{prop:exp} for the MIET of the \pbb design with
exponential specification does not depend on the state, and can
therefore be calculated a priori, before the actual implementation of
the controller.  We take advantage of the ability to quantify the
benefits of the \pbb approach for exponential specifications when
discussing its application to network systems in our forthcoming
discussion.

\section{\PBB Triggering for Network
  Systems}\label{sec:dist}

In this section we discuss the application of the \pbb
triggering approach to the design of distributed triggers for network
systems. Specifically, we consider exponential performance
specifications and take advantage of the additional flexibility
provided by the performance residual to ensure the existence of
a~MIET.

Consider a network of $N$ agents whose interconnection is represented
by a connected undirected graph $\mathcal G = (\until N,\mathcal
E)$. By this, we mean that each agent can only communicate with its
neighbors, and hence has access to limited information about the
system.  We make the assumption that the Lyapunov function
satisfying~\eqref{eq:ExpLyap} can be expressed as an aggregate
\begin{align*}
  V(x) = \sum_{i=1}^N V_i(x_{\NN_i}) ,
\end{align*}
with each function $V_i$ depending on the local information available
to agent~$i$.  We assume each $V_i$ to be continuously differentiable
with Lipschitz gradient.  Our goal is to design distributed triggers
that can be evaluated by individual agents with the information
available to them.

\subsection{Challenges for ETC in Network
  Systems}\label{sec:challenges-ETC-network}
Here we describe the challenges in transcribing the derivative-based
trigger approach to network systems.  The direct transcription
of~\eqref{Trigger:deriv} to the network setting would result in a
centralized trigger that requires global information to be
evaluated. Making use of the aggregate decomposition of $V$, one can
instead define
\begin{multline}\label{eq:naive-distributed}
  t_{k+1} = \min \setdefb{t\geq t_k}{ \exists i\in \until N~\ni
    \\
    (\sigma-1)c_\alpha\|x_i(t)\|^2 + c_\gamma\|e_i(t)\|^2 +r
    V_i(x_{\NN_i}(t))\geq 0}.
\end{multline}
Note the slight abuse of notation here, where $x_i$ and $e_i$ now
refer to the states associated with agent $i$, rather than the $i$-th
component of vectors $x$ and~$e$, resp.  This trigger corresponds to
partitioning~\eqref{Trigger:deriv} across the network into multiple
triggers, one per agent, that can be individually evaluated with local
information.  Note that the design means that when an agent triggers,
a controller update request is sent network-wide.  This relies on the
observation that such messages, which do not require any state
information, can be easily propagated through the network.  The design
is more conservative than the centralized one and, as a consequence,
results in shorter inter-event times for an arbitrary network
state. In fact, this type of distributed trigger schemes can suffer
from Zeno behavior, see
e.g.,~\cite{WPMHH-KHJ-PT:12,MCFD-WPMH:12,PT-NC:14,DPB-WPMHH:14}.
A common practice to address this is to explicitly incorporate a MIET
at the design stage, a process known as time regularization, see
e.g.,~\cite{DPB-VSD-WPMHH:16,MMJ-PT:10,PT-NC:14}.  For instance, with
a slight modification to suit our context,~\cite{MMJ-PT:10} proposes
the following
\begin{align}\label{Trigger:dist-Paulo}
  &t_{k+1}  = \min \setdefb{t\geq t_k+\tau^{\deriv}_\sigma}{ \exists
    i\in \until N~\ni
    \\
    & \; \; (\sigma-1)c_\alpha\|x_i(t)\|^2 + c_\gamma\|e_i(t)\|^2 +r
    V_i(x_{\NN_i}(t))\geq b_i(t_k)}, \notag
\end{align}
where $b\in \real^n$ is a budget variable satisfying $\ones^\top b =
0$, which we discuss below.  Time regularization discards the
possibility of Zeno behavior by forcing the inter-event time to be
above the MIET known from the centralized design.  The design builds
on the fact that, from the analysis in Section~\ref{sec:nonlin}, we
know controller updates are not necessary for $\tau^{\deriv}_\sigma$
seconds after the last update in order to meet the performance
specification.  Consequently, agents can ignore the trigger conditions
for this amount of time and only start enforcing them thereafter.

However, note that time regularization does not change the fact that
the error $\|e_i\|$ might have already surpassed the level at which
the trigger would occur as soon as the trigger condition starts
getting monitored, see e.g.,~\cite{DPB-WPMHH:14}.  The variable $b$
seeks to address this by re-balancing the budget that each agent has
in its trigger condition, allowing for the possibility of allocating
at the triggering times some budget from a node where the condition
has not been violated to another node where it has (in order to have
the latter not trigger immediately next time once
$\tau^{\deriv}_\sigma$ seconds have elapsed).
Among the potential disadvantages of the
design~\eqref{Trigger:dist-Paulo} from a network perspective, we point
out the following:
\begin{enumerate}
\item the computation of the MIET~$\tau^{\deriv}_\sigma$ can be
  challenging and requires the execution of a dedicated distributed
  algorithm prior to the controller implementation.  Moreover, the
  value obtained may turn out to be too conservative, making the
  trigger occur more frequently than necessary;
\item the proposed scheme requires a central entity, albeit only at
  each triggering time, to calculate and assign budgets to all the
  agents;
\item without further assumptions on the nonlinear system, the
  evolution of the trigger condition cannot be predicted, and
  consequently there is no guarantee that the selected budgets~$b$
  will successfully extend the inter-event time.
\end{enumerate}
Our proposed method addresses these problems by designing a trigger
that intrinsically exhibits a MIET and relying on distributed
computation and communication among the agents to calculate their
budgets.

\subsection{Intrisically Zeno-Free Distributed ETC Design}

We use two different elements to propose a distributed trigger scheme:
dynamic average consensus algorithm and the \pbb trigger
design. We approach the Zeno problem by attacking directly its root
cause in distributed settings: partial information of the system
states is insufficient to inform agents of system's overall
performance.
For this reason, our distributed trigger design makes use of dynamic
average consensus algorithm to estimate, with some tracking errors,
the global terms in the centralized version of the trigger.  Doing so
transforms the problem into ascertaining how well the trigger design
can tolerate errors.  This is where we leverage the additional
flexibility provided by the \pbb approach over the
derivative-based one regarding handling of the tracking
errors. Particularly, as we will show later in the analysis of our
design, the performance residual term offered by \pbb ETC
plays a key role in ruling out Zeno behavior.

We begin by defining some notation functions for compactness of
presentation. Let
\begin{align*}
  \mathcal W^{x}(x) &= (\sigma-1)c_\alpha\|x\|^2 +(r+c_\beta)V(x),
  \\
  \mathcal W^{xe}(x,e) &= \mathcal W^{x}(x) + c_\gamma\|e\|^2.
\end{align*}
These functions can be decomposed as sums of the following functions,
respectively,
\begin{align*}
  \Wc^x_i(x_{\mathcal N_i}) &= (\sigma-1)c_\alpha
  \|x_i\|^2+(r+c_\beta)V_i(x_{\mathcal N_i}) ,
  \\
  \Wc^{xe}_i(x_{\mathcal N_i},e_i)&= \Wc^x_i(x_{\mathcal N_i})+
  c_\gamma\|e_i\|^2 .
\end{align*}    
For convenience, we let $W^{x}$ and $W^{xe}$ be vector-valued
functions with components $W^x_i = \Wc^x_i $ and $W^{xe}_i =
\Wc^{xe}_i $, respectively. We omit the dependency on $x_{\mathcal
  N_i}$ and $e_i$ when it is clear from the context. Notice that
$\ones^\top W^{x} = \mathcal W^{x}$ and $\ones^\top W^{xe} = \mathcal
W^{xe}$. The centralized \pbb trigger
design~\eqref{Trigger:barr} can be rewritten compactly as
\begin{equation}\label{Trigger:cent}
  t^{\exp}_{k+1} = \min \setdefb{t\geq t_k}{\mathcal
    W^{xe}(t)= c_\beta V(x_0)\exp(-rt)}. 
\end{equation}
This trigger has a MIET, cf. Proposition~\ref{prop:exp}, but the
direct computation of $\Wc^{xe}$ requires global information.
However, given the aggregate decomposition $\ones^\top W^{xe} =
\mathcal W^{xe}$ and the fact that agent $i$ knows $\Wc^{xe}_i$, a
dynamic average consensus algorithm enables the agents to estimate the
average $\Wc^{xe}/N$. This leads to the following trigger design,
\begin{subequations}\label{Trigger:dist}
  \begin{align}
    t_{k+1} & = \min \setdefb{t\geq t_k}{\exists i\in \until N~\ni
      \notag
      \\
      & \quad \qquad a_i(t) =c_\beta V(x_0)\exp(-rt)/N}, \label{Trigger:dist_condition}
    \\
    \dot a & = \dot W^{xe} -\rho_a \mathbf{L}a, \label{sys:dyncon}
  \end{align}
  where $\rho_a > 0$ and $\mathbf{L}\in\real^{N\times N}$ is the graph's
  Laplacian.  With this formulation, we denote the tracking error by
  $\epsilon_a := a-\ones\ones^\top W^{xe}/N$.  In order for the dynamic
  average consensus to track the right variable, it is crucial to
  initialize $a$ so that $\ones^\top \epsilon_a = 0$.
  As such, we assume that $a(0)$ is so that $\epsilon_a(0) = 0$ at the
  initial time $t=0$. Since the tracking error's mean $\ones^\top
  \epsilon_a$ is conserved along the dynamics~\eqref{sys:dyncon}, this
  ensures $\ones^\top \epsilon_a = 0$ until the next triggering time.
  However, the value of $\mathcal W^{xe}$ jumps to $\mathcal W^{x}$ at each trigger time
  $t_k$ due to $e$ being reset to zero, and therefore the average
  estimate $a$ must be reinitialized at each trigger time $t_k$ to
  keep the tracking error's mean zero. To do this, we use another
  dynamic average consensus to keep track of $\mathcal W^{x}$ as
  \begin{equation}\label{sys:dyncon_init}
    \dot z = \dot W^x - \rho_z\mathbf{L}z
  \end{equation}
  where $\rho_z>0$, with the initial condition $z(0) = \ones \mathcal
  W^x(x_0)/N$. Similarly, we denote the tracking error by~$\epsilon_z
  := z-\ones\ones^\top W^{x}/N$. Note that the variable $z$ does not
  depend on $e$, so it does not need to reinitialize at each $t_k$.
  With the new tracking variable, we reinitialize $a$ to $z$ at each
  trigger time with a jump map,
  \begin{equation}\label{eq:jump}
    a^+ = z,~t\in \{t_k\}_{k=0}^\infty.
  \end{equation}
\end{subequations}


\begin{remark}\longthmtitle{Distributed Implementation} {\rm The
    design~\eqref{Trigger:dist} does not require a central entity to
    estimate the evolution of the trigger condition, relying instead
    on dynamic average consensus.  To implement~\eqref{Trigger:dist},
    the $i$-th agent, with local exchange information on $a_i$ and
    $z_i$, can evaluate the dynamic average consensus
    dynamics~\eqref{sys:dyncon} and \eqref{sys:dyncon_init} if the
    time derivative of the reference signals $\dot W_i^{xe}$ and $\dot
    W_i^{x}$ are available to it. Each agent $i$ has the information
    of the states $x_i$ and $e_i$ and the dynamics $\dot x_i$ and
    $\dot e_i$. However, due to dependency on $x_{\NN_i}$, the
    calculation of $\dot W_i^{xe}$ and $\dot W_i^{x}$ requires
    knowledge of $x_{\NN_i}$ and $\dot x_{\NN_i}$. The computation of
    the latter requires two-hop communication in the graph
    (alternatively, only one-hop communication is required if the
    decomposition of the Lyapunov function takes the form $ V(x) =
    \sum_{i=1}^N V_i(x_{i})$).  \hfill~$\bullet$}
\end{remark}


\begin{remark}\longthmtitle{Extensions to Discrete-Time Consensus and
    Directed Graphs}\label{rmk:discrete}
  {\rm Instead of the continuous-time algorithms in~\eqref{sys:dyncon}
    and~\eqref{sys:dyncon_init}, the design~\eqref{Trigger:dist} could
    employ discrete-time implementations of the dynamic average
    consensus algorithm, see
    e.g.,~\cite{SSK-BVS-JC-RAF-KML-SM:19}. Since the effective
    timescales of~\eqref{sys:dyncon} and~\eqref{sys:dyncon_init} scale
    linearly with $\rho_a$ and $\rho_z$, respectively,
    cf. Lemma~\ref{lem:trackbound}, the stepsizes of such
    discrete-time implementations would scale linearly with $1/\rho_a$
    and $1/\rho_z$, respectively. A technical analysis analogous to
    the one presented in Section~\ref{sec:convergence-ETC} below could
    be developed, albeit we do not pursue it here for simplicity of
    exposition. A similar observation can be made about the
    interconnection structure of the network, which could easily be
    extended from undirected to weight-balanced, strongly connected
    directed graphs, cf.~\cite{FB-JC-SM:08cor}.  \hfill~$\bullet$}
\end{remark}

\subsection{Convergence Analysis}\label{sec:convergence-ETC}

In this section we show that the proposed distributed trigger
design~\eqref{Trigger:dist}, with suitable choices of the parameters
$c_\beta$, $\rho_a$, and $\rho_z$, makes the origin asymptotically
stable.  Our analysis includes establishing performance satisfaction
and a MIET. Regarding the former, from the definition of the trigger,
we have that
\begin{equation}\label{eq:dist_enforce}
  \Wc^{xe}(t)/N + \epsilon_{a,i}(t) = a_i(t) < c_\beta V(x_0)\exp(-rt)/N
\end{equation}
along the trajectory for all $i \in \until{N}$. Using the fact that
$\ones^\top\epsilon_a =0$ at all time and
summing~\eqref{eq:dist_enforce}, we deduce that $\Wc^{xe}(t)< c_\beta
V(x_0)\exp(-rt)$, i.e., the same condition enforced by the centralized
trigger~\eqref{Trigger:cent}. This shows the satisfaction of
performance. Establishing MIET is more complicated. The
inequality~\eqref{eq:dist_enforce} suggests that $\epsilon_{a,i}$
being nonzero can make the distributed trigger~\eqref{Trigger:dist}
occur prematurely in comparison to the centralized
trigger~\eqref{Trigger:cent}. However, our analysis below shows that,
by tuning different parameters appropriately, we can ensure that at
least for the time interval $[t_k,t_k +\tau^{\deriv}_\sigma)$, the
presence of $\epsilon_{a,i}$ does not have this effect, and
\eqref{Trigger:dist} is not triggered. Before establishing this fact,
we show next that the reference signals $W^{xe}$ and $W^{x}$ have an
exponentially bounded time derivative. Its proof is given in the
appendix.

\begin{lemma}\longthmtitle{Exponential Bounds for Reference
    Signals}\label{lem:signal-exp}
  Consider the distributed trigger design~\eqref{Trigger:dist} for the
  sample-and-hold nonlinear system~\eqref{sys:NL} with Lipschitz $F$
  and $\kappa$. Assume that each $V_i$ is continuously differentiable
  with Lipschitz gradient. Given a desired rate of convergence
  $r<c_\alpha/c_2$ and $\sigma\in(0,1-\frac{rc_2}{c_\alpha})$, there
  exists $\Omega^{xe}>0$ such that, for all $k\in \{0\}\cup
  \naturals$,
  \begin{align*}
    \|\dot W^{xe}(t)\| \leq \Omega^{xe}V(x_k)\exp(-r\Delta t_k)
  \end{align*}
  for $t\in[t_k,t_k+\tau^\deriv_\sigma)$.  Furthermore if $c_\beta >
  (1-\sigma)(c_\alpha/c_1)- r$, there exists $\Omega^{x}>0$ such that
    $$
  \|\dot W^{x}(t)\| \leq \Omega^{x}V(x_0)\exp(-rt)
  $$
  for all time along the trajectory.~\hfill \qed
\end{lemma}


Lemma~\ref{lem:signal-exp} ensures that the requirements to apply
Lemma~\ref{lem:trackbound} hold, allowing us to bound $\epsilon_a$
and~$\epsilon_z$.  We are now ready to state the main result of this
section.

\begin{theorem}\longthmtitle{Distributed ETC
    with Exponential Performance}\label{thm:dist}
  Consider the sample-and-hold nonlinear system \eqref{sys:NL} with a
  Lyapunov function~\eqref{eq:ExpLyap}. Given a desired rate of
  convergence $r<c_\alpha/c_2$ and
  $\sigma\in(0,1-\frac{rc_2}{c_\alpha})$, let $t_{k+1}$ be determined
  iteratively according to the \pbb distributed
  trigger~\eqref{Trigger:dist} with $c_\beta >
  (1-\sigma)(c_\alpha/c_1)-r$.   Under the assumption that $F$ and
  $\kappa$ are Lipschitz and that each $V_i$ is continuously
  differentiable with Lipschitz gradient, let  
  the constant $\tau_\sigma^\deriv$ be defined as in
  Proposition~\ref{prop:exp}. Then, there exist $\rho_a$ and
  $\rho_z$ large enough such that
  \begin{align*}
    t_{k+1}-t_k\geq \tau^\deriv_\sigma ,
  \end{align*}
  for all $k\in \{0\}\cup \naturals$.  Consequently, the performance
  requirement $V(x(t))\leq V(x_0)\exp(-rt)$ is enforced for all time
  and the origin is rendered globally exponentially stable.
\end{theorem}
\begin{proof}
  Our proof strategy is to show that, for each $k \in \{0\}\cup
  \naturals $, $\max_{i \in \until{N}} a_i-c_\beta
  V(x_0)\exp(-rt)/N<0$ during the time period
  $[t_k,t_k + \tau^\deriv_\sigma)$, which implies that no trigger occurs
  in said period. Note the bound
  \begin{align*}
    \max_{i\in \until N} a_i = \max_{i \in \until N} \epsilon_{a,i} +
    \Wc^{xe}/N \leq \|\epsilon_a\|+\mathcal W^{xe} /N.
  \end{align*}
  Therefore, it is enough to prove instead that
  \begin{equation}\label{eq:track-err-allowed}
    \|\epsilon_a\|+ \frac{1}{N} \big( \Wc^{xe}  -c_\beta
    V(x_0)\exp(-rt) \big)<0.
  \end{equation}
  We start bounding the second summand.  Using the bounds $\|e\|\leq
  \phi(t-t_k)\|x\|$ from Lemma~\ref{lem:errBnd} and $\|x\|^2\geq
  V(x)/c_2$ from~\eqref{eq:ExpLyapA},
  \begin{align*}
    \mathcal W^{xe} & \leq
    \left((\sigma-1)c_\alpha+c_\gamma\phi(\Delta
      t_k)^2\right)\|x\|^2+(r+c_\beta)V(x)
    \\
    & \leq \left(\frac{(\sigma-1)c_\alpha+c_\gamma\phi(\Delta
        t_k)^2}{c_2}+r+c_\beta\right)V(x)
    \\
    & = (\xi(\Delta t_k)+r+c_\beta)V(x)
  \end{align*}
  for $t\in [t_k,t_k+\tau^\deriv_\sigma]$.  Notice from the second
  inequality that with $c_\beta>(1-\sigma)(c_\alpha/c_1) -r$, the
  coefficient of $V(x)$ is positive, so we can use the upper bound of
  $V$ from \eqref{eq:gronLyap} to get
  \begin{align*}
    \mathcal W^{xe} -&c_\beta V(x_0)\exp(-rt)
    \\
    &= \mathcal W^{xe} -c_\beta V(x_k)\exp(-r\Delta t_k)
    \\
    & \quad -c_\beta(V(x_0)\exp(-rt)-V(x_k)\exp(-r\Delta t_k))
    \\
    & \leq ~(\xi(\Delta t_k)+r)V(x_k) \exp \Big(\int_{0}^{\Delta
      t_k}\xi(s)ds \Big)
    \\
    &\quad -c_\beta V(x_k) \Big(\exp(-r\Delta t_k)-\exp
    \Big(\int_{0}^{\Delta t_k}\xi(s)ds\Big)\Big)
    \\
    & \quad -c_\beta(V(x_0)\exp(-rt)-V(x_k)\exp(-r\Delta t_k)).
  \end{align*}
  Consider the first two terms in this expression. Both terms are
  strictly negative in the time interval $[t_k,t_k
  +\tau^\deriv_\sigma)$, so the maximum value of their sum must be
  negative. Therefore, there exists $\Omega^*>0$ (which can be found
  explicitly by examining its derivative and endpoints on the time
  interval $\Delta t_k\in [0,\tau^\deriv_\sigma]$ ) independent of $x_k$ such that
  \begin{align}
    &\mathcal W^{xe} -c_\beta
    V(x_0)\exp(-rt) \label{eq:2nd_summand_bound}
    \\
    &\leq -\Omega^*
    V(x_k)-c_\beta\big(V(x_0)\exp(-rt)-V(x_k)\exp(-r\Delta t_k)\big)
    \nonumber
    \\
    &= -\Omega^* V(x_k)-
    c_\beta\Big(V(x_0)\exp(-rt_k)-V(x_k)\Big)\exp(-r\Delta t_k)
    \nonumber
    \\
    &\leq -\Omega^* V(x_k)-
    c_\beta\big(V(x_0)\exp(-rt_k)-V(x_k)\big)\exp(-r
    \tau^\deriv_\sigma). \nonumber
  \end{align}
  Note here that both terms in the bound are non-positive.
  
  Regarding the first summand $\|\epsilon_a\|$
  in~\eqref{eq:track-err-allowed}, we resort to
  Lemma~\ref{lem:trackbound} to bound it.  We
  write~\eqref{eq:trackbound}, with a change of variable to shift
  time by $t_k$, for $\epsilon_a$,
  \begin{multline*}
    \|\epsilon_a(t)\| \leq
    \frac{\Omega^{xe}V(x_k)}{\rho_a\lambda_2-r} \exp(-r\Delta t_k)
    \\
    +\left(\|\epsilon_a(t_k)\| -
      \frac{\Omega^{xe}V(x_k)}{\rho\lambda_2-r}\right)\exp(-\rho
    \lambda_2 \Delta t_k).
  \end{multline*}
  Over the time interval $\Delta t_k \geq 0$, the bound either
  achieves the maximum value at $\Delta t_k=0$ or where its time
  derivative is zero on the positive interval $\Delta t_k >0$. In
  other words, $\|\epsilon_a(t) \| \leq \max
  \{\|\epsilon_a(t_k)\|,\frac{\Omega^{xe} V(x_k)}{\rho_a
    \lambda_2-r}\}$.
  We consider these two scenarios separately.

  First, consider the case where the $\|\epsilon_a(t_k)\|\leq
  \frac{\Omega^{xe} V(x_k)}{\rho_a \lambda_2-r} $. By selecting
  $\rho_a > (1/\lambda_2)(N\Omega^{xe}/\Omega^*+r)$, we can ensure
  that $\frac{\Omega^{xe} V(x_k)}{\rho_a \lambda_2-r} <
  \Omega^*V(x_k)/N$. This shows that the first term in the upper
  bound~\eqref{eq:2nd_summand_bound} is enough to dominate
  $\|\epsilon_a(t)\|$, guaranteeing that \eqref{eq:track-err-allowed}
  holds.
  
  Next, consider the case where $\|\epsilon_a(t_k)\|>\frac{\Omega^{xe}
    V(x_k)}{\rho_a \lambda_2-r}$.  Because $W^{xe}(t_k)=W^{x}(t_k)$
  holds at the update time $t_k$, we deduce from the jump
  map~\eqref{eq:jump} that $\epsilon_a(t_k) = \epsilon_z(t_k)$. Thus,
  the size of $\|\epsilon_a(t_k)\|$ directly depends on how well the
  dynamic average consensus~\eqref{sys:dyncon_init} performs, so we
  tune $\rho_z$ appropriately so that \eqref{eq:track-err-allowed}
  holds.  Particularly, we look at the possibility that
  \begin{align*}
    \|\epsilon_z(t_k)\| \geq c_\beta(V(x_0)\exp(-rt_k)-V(x_k))\exp(-r
    \tau^\deriv_\sigma)/N      
  \end{align*}
  (otherwise, the second term of the upper
  bound~\eqref{eq:2nd_summand_bound} already dominates
  $\|\epsilon_a(t)\|$).
  %
  From~\eqref{eq:trackbound}, and given the initialization of $z$ with
  $\epsilon_z(0)=0$, we have
  \begin{align*}
    \|\epsilon_z(t_k)\| \leq \frac{\Omega^x
      V(x_0)}{\rho_z\lambda_2-r}\big(\exp(-rt_k) -
\exp(-\rho_z\lambda_2t_k) \big).
  \end{align*}
   Since $\exp(-\rho_z\lambda_2t_k)\geq 0$, we obtain the relationship
  \begin{multline}\label{eq:case-4_relationship}
    \frac{\Omega^x V(x_0)}{\rho_z\lambda_2-r}\exp(-rt_k) \geq \|\epsilon_z(t_k)\|\\
    \geq
    c_\beta(V(x_0)\exp(-rt_k)-V(x_k))\exp(-r\tau^\deriv_\sigma)/N.
  \end{multline}
  After some algebraic manipulations, this implies
  \begin{align*}
    V(x_0) \exp(-rt_k) \leq
    \frac{c_\beta\exp(-r\tau^\deriv_\sigma)}{c_\beta\exp(-r\tau^\deriv_\sigma)
      - \frac{N\Omega^x}{\rho_z\lambda_2-r}} V(x_k) ,
  \end{align*}
  if the denominator of the right-hand side is positive. For this to
  be the case, we have to make sure that our choice of $\rho_z$
  satisfies $\rho_z>
  (1/\lambda_2)(\frac{N\Omega^x}{c_\beta\exp(-r\tau^\deriv_\sigma)}+r)$.
  Substituting the bound above into the upper bound
  in~\eqref{eq:case-4_relationship}, we get
  \begin{align*}
    \|\epsilon_z(t_k)\| \leq \frac{\Omega^x
      c_\beta\exp(-r\tau^\deriv_\sigma)}{c_\beta\exp(-r\tau^\deriv_\sigma)(\rho_z\lambda_2-r)
      - N\Omega^x} V(x_k).
  \end{align*}
  Now, any selection of $\rho_z$ such that
  \begin{align*}
    \rho_z > \frac{1}{\lambda_2}\left(\frac{N\Omega^x}{\Omega^*} +
      \frac{N\Omega^x}{c_\beta\exp(-r\tau^\deriv_\sigma)}+r\right),
  \end{align*}
  ensures that
  \begin{align*}
    \frac{\Omega^x
      c_\beta\exp(-r\tau^\deriv_\sigma)}{c_\beta\exp(-r\tau^\deriv_\sigma)(\rho_z\lambda_2-r)
      - N\Omega^x} < \frac{\Omega^*}{N},
  \end{align*}
  and therefore $\|\epsilon_z(t_k)\| <\Omega^*V(x_k)/N$, implying that
  the first term of the upper bound in~\eqref{eq:2nd_summand_bound}
  dominates $\|\epsilon_a(t)\|$.
  Therefore, \eqref{eq:track-err-allowed} holds for
  $t\in[t_k,t_k+\tau^\deriv_\sigma)$, and
  $\tau^\deriv_\sigma$ is a MIET for the distributed trigger
  design~\eqref{Trigger:dist}. With the existence of the MIET,
  performance satisfaction and global exponential stability follow.
\end{proof}

Theorem~\ref{thm:dist} shows that, with the appropriate tuning of the
design parameters,~\eqref{Trigger:dist} is an intrinsically Zeno-free
event-triggered design for network systems with exponential
performance (without the need to prescribe the MIET in the design as
in~\eqref{Trigger:dist-Paulo}).
This property relies critically on the \pbb design
approach, particularly on the robustness to errors provided by the
performance residual.

\begin{remark}\longthmtitle{Conservativeness in Design Parameters}
  {\rm The required bounds for the design parameters $\rho_a$ and
    $\rho_z$ developed in the proof of Theorem~\ref{thm:dist} are
    conservative and, in fact, we have observed in practice that
    values that violate these bounds also result in successful
    executions.
    Such bounds must be computed offline, a requirement that is also
    shared by the time-regularization method regarding the computation
    of the~MIET.  However, the key difference, beyond the fact that
    the method proposed here overcomes the challenges~(i)-(iii)
    described in Section~\ref{sec:challenges-ETC-network}, is that
    conservativeness in the MIET computation leads to higher actuation
    resource usage, whereas conservativeness in the bounds of
    Theorem~\ref{thm:dist} imposes requirements on the communication
    and computational resources of the agents, without affecting the
    timing of the triggers.
\hfill~$\bullet$}
\end{remark}


\section{Simulations on Vehicle Platooning}\label{sec:sims}
To illustrate the effectiveness of the \pbb trigger
design approach, we consider a vehicle platooning problem with $N=5$
vehicles driving in a line formation along a rectilinear curve.
Following~\cite{VSD-JP-WPMHH:17}, we seek to take advantage of the
inter-agent communication resources to minimize the usage in actuation
resources. The goal is to synchronize the speed $v_i$ of each vehicle
$i \in \{2,\dots,5\}$ to the leader's desired speed $v_\des$, and the
vehicle's following distance $d_i$ to a safe distance
$d_{\des,i}=d_0+T_vv_i$. Here, $d_0$ is the standstill following
distance and $T_v$ represents the factor for the additional distance
to keep with respect to the vehicle's speed. Vehicle $1$ is the leader
and measures distance with respect to a virtual reference vehicle.  We
define $\delta_i:= d_i-d_{\des,i}$ and $\nu:= v_i-v_\des$ to be the
mismatch between the actual and the desired variables. Each vehicle
uses a dynamic feedback controller to compute its control input $u_i$,
which directly affects the vehicle's acceleration~$q_i$.  The
closed-loop dynamics of the leading car, with state $x_i
= \begin{bmatrix} \delta_i & \nu_i & q_i & u_i \end{bmatrix}^\top$,
can be written as
\begin{align*}
  \dot x_1 &= \begin{bmatrix}
    0 && -1 && -T_v && 0 \\
    0 && 1 && 0 && 0 \\
    0 && 0 && -\frac{1}{T_d} && \frac{1}{T_d}\\
    \frac{k_p}{T_v} && -\frac{k_d}{T_v} && -k_d && -\frac{1}{T_v}
  \end{bmatrix} x_1 + \begin{bmatrix} 0 \\ 0 \\ \frac{1}{T_d} \\ 0
  \end{bmatrix}e_1
  \\
  &= \bar A_\diag x_1 + \bar E e_1,
\end{align*}
where $e_1=x_{1,4}(t_k)-x_{1,4}$ encodes the fact that the actual
control input $u$ is sampled at time $t_k$ and held constant until
$t_{k+1}$.  We use the system parameters $k_p =0.2$, $k_d =0.7$, $T_v
=0.6$, and $T_d = 0.1$.  Vehicles $\{2,\dots, 5\}$ have dynamics that
depend on the cars in front of them, as follows
\begin{align*}
  \dot x_i & = \bar A_\diag x_i + \bar A_\off x_{i-1} + \bar E e_i,
\end{align*}
where $e_i = x_{i,4}(t_k)-x_{i,4}$ is the sample-and-hold error and
\begin{align*}
  \bar A_\off & =
  \begin{bmatrix}
    0 && 1 && 0 && 0 \\
    0 && 0 && 0 && 0 \\
    0 && 0 && 0 && 0 \\
    0 && \frac{k_d}{T_v} && 0 && \frac{1}{T_v}
  \end{bmatrix} .
\end{align*}
We next explain how we obtain an ISS Lyapunov function. First, we
find $\Pc>0$ such that $\bar A_\diag^\top \Pc+\Pc \bar A_\diag =
-\mathbb I$ (this corresponds to ignoring the interconnection of each
following vehicle with the one in front). Next, we define
\begin{align*} 
  V(x)=\sum_{i=1}^N \pi^{N-i}x_i^\top \Pc x_i ,
\end{align*}
where $\pi$ is a weight factor to be chosen. Note that this definition
naturally places more weight to the vehicles towards the front of the
platoon. The Lie derivative of $V$ is given by
\begin{align*}
  L_fV(x,e) =&~ \sum_{i=2}^N\pi^{N-i}(-\|x_i\|^2+ 2x_i^\top \Pc\bar
  A_{\off}x_{i-1})
  \\
  &-\pi^{N-1}\|x_1\|^2 +\sum_{i=1}^N \pi^{N-i}2x^\top_i\Pc\bar E e_i .
\end{align*}
Using Young's inequality~\cite{GHH-JEL-GP:52}, we can bound the cross
terms as $2x_i^\top P\bar A_{\off}x_{i-1}\leq 5\|\Pc\bar
A_{\off}\|^2\|x_{i-1}\|^2+(1/5)\|x_{i}\|^2$.
Selecting then $\pi=31.25\|\Pc\bar A_{\off}\|^2$, we find, after some
algebraic manipulations, that
\begin{align*}
  L_fV(x,e) \leq -0.145V(x) +\sum_{i=1}^N \pi^{N-i}2x^\top_i\Pc\bar E e_i.
\end{align*}
This implies a rate of convergence of $r^*=0.145$ in the absence of
sample-and-hold error~$e$. In our simulations, we specify the desired
exponential convergence rate $r=0.08<0.75r^*$ for the triggered
implementations.

With all the elements in place, we are ready to provide a comparison
of different event-triggered control approaches.  We implement the
centralized \pbb trigger design, specifically the linear
one in~\eqref{Trigger:barr-lin}, and compare it to the
derivative-based design~\eqref{Trigger:deriv-lin}. For this, we use
\begin{align}\label{eq:g}
  g(x,e) =&~ 0.75\sum_{i=2}^N\pi^{N-i}(-\|x_i\|^2+ 2x_i^\top \Pc\bar
  A_{\off}x_{i-1}) \notag
  \\
  &-0.75\pi^{N-1}\|x_1\|^2 +\sum_{i=1}^N \pi^{N-i}2x^\top_i\Pc\bar E e_i,
\end{align}
and $c_\beta=1$. Each simulation lasts 400
seconds. Figure~\ref{fig:Lyap} shows the evolution of the Lyapunov
functions in logarithmic scale for different trigger designs and
Table~\ref{tab:results} shows the empirical MIET (which might be
larger than the actual MIET) and average number of controller updates
across 50 different trajectories with random initial conditions. As
expected, both designs satisfy the required performance.
However, it is evident from Figure~\ref{fig:Lyap} that the
derivative-based design outperforms the requirement, meaning that the
number of updates could be significantly reduced. This is precisely
what the \pbb design accomplishes by tuning the timing of
the updates to the degree of satisfaction of the prescribed
performance, reducing their number by almost 20-fold on average.

\begin{figure}[tbh]
  \centering
  \includegraphics[width=\linewidth,height=\textheight,keepaspectratio]{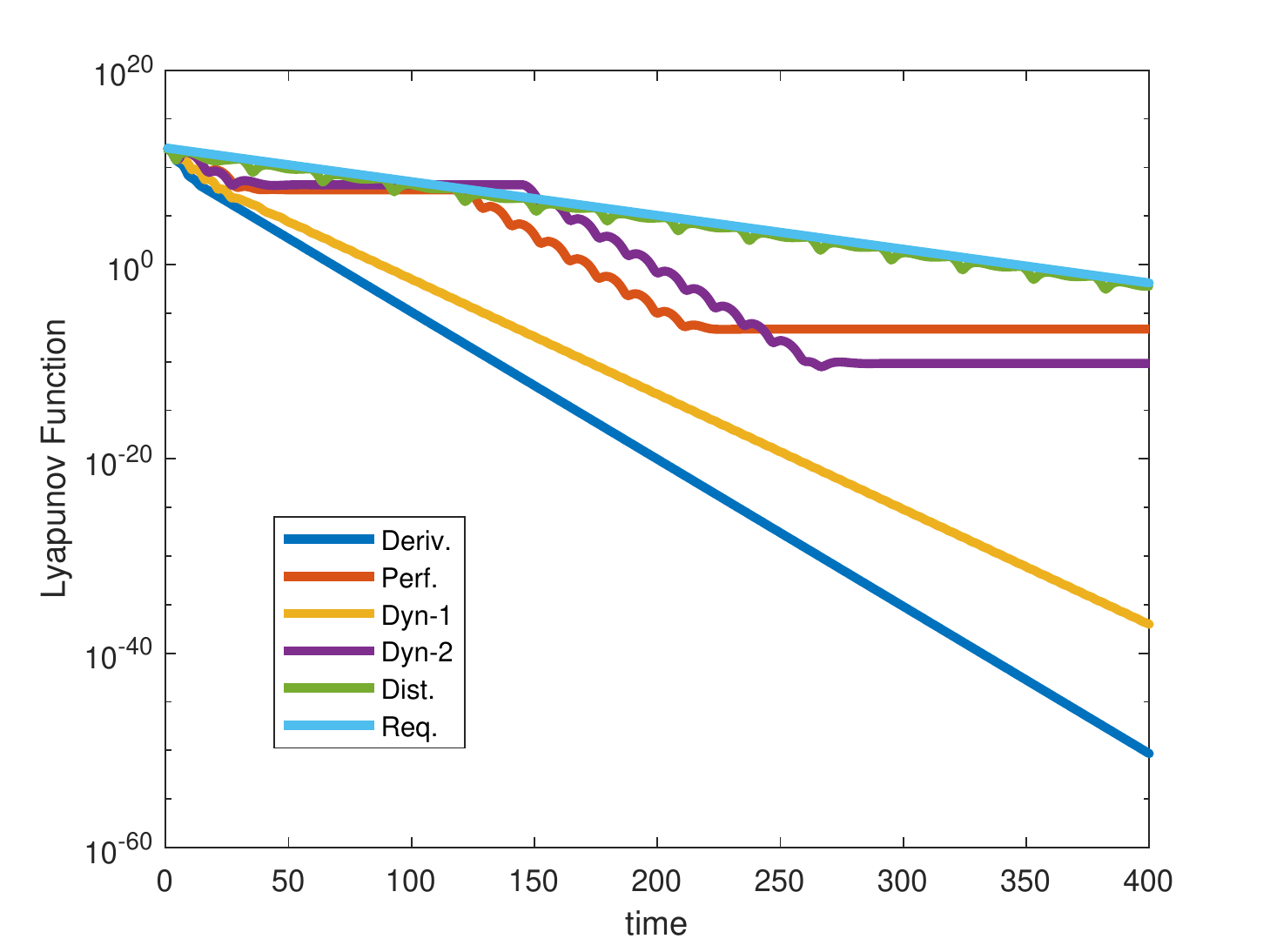}
  \caption{Evolution of the Lyapunov function for different trigger designs. }\label{fig:Lyap}
\end{figure}

\begin{table}[tbh]
  \centering
  \caption{Empirical MIET and average number of updates from 50 different random initial conditions}\label{tab:results}
  \begin{tabular}{||l|c|c||}
    \hline 
    Design & MIET (s) &  Avg. no. of updates \\
    \hline \hline
    Derivative-Based \eqref{Trigger:deriv-lin} & 0.009 & 198.28 \\
    \hline
    \PBB \eqref{Trigger:barr-lin} & 0.009 & 9.94 \\
    \hline
    Dynamic \eqref{eq:dynamic} & 0.013 & 108.48 \\
    \hline
    Dynamic \eqref{eq:dynamic} -- small decay & 0.013 & 7.18 \\
    \hline
    Distributed \eqref{Trigger:dist} & 0.003 & 96.38 \\
    \hline
  \end{tabular}
\end{table}

We also compare the proposed designs with the dynamic trigger
approach, cf.~Remark~\ref{re:dynamic-trigger}. To do so, we choose a
linear decay function $\iota (\eta) = c_\iota \eta$, and consider
different values of~$c_\iota$.  According to~\eqref{eq:dynamic}, the
degree of decay of the Lyapunov function $V$ grows with the value of
$c_\iota$, but it is not possible to determine in advance whether a
given value of $c_\iota$ will guarantee that the evolution meets the
desired performance specification. We first use $c_\iota =1$, and
observe, cf. Figure~\ref{fig:Lyap}, that the evolution of $V$ is,
similarly to that of the derivative-based design, too conservative.
Consequently, we employ $c_\iota = 0.05$, which leads to a significant
decrease in the number of updates, cf. Table~\ref{tab:results}, at the
cost of not meeting the performance specification any more,
cf. Figure~\ref{fig:Lyap}. One could go through the exercise of
fine-tuning the value of $c_\iota$ to make sure the trajectories meet
the desired performance, but this would have to be verified a
posteriori in an empirical way, rather than a priori by design, as the
\pbb approach does.

Lastly, we also report the simulation results of the distributed
trigger design~\eqref{Trigger:dist} with $\rho_a =10$ and $\rho_z=20$.
In fact, notice that both the Lyapunov function $V$ and $g$
in~\eqref{eq:g} can be expressed as the sum of functions, one per
agent, whose value can be computed by each agent with local
information,
\begin{align*}
  &V_i(x_i) = \pi^{N-i} x_i^\top\Pc x_i,\qquad W_1^{x}(x_1) =
  -0.75\pi^{N-1}\|x_1\|^2, 
  \\
  &W_i^{x}(x_{\mathcal N_i}) = 0.75\pi^{N-i}(-\|x_i\|^2+ 2x_i^\top
  \Pc\bar
  A_{\off}x_{i-1}),~\forall i\geq 2
  \\
  &W_i^{xe}(x_{\mathcal N_i},e_{i}) = W_i^{x}(x_{\mathcal N_i})+
  \pi^{N-i}2x^\top_i\Pc\bar E e_i ,~\forall i\geq 1.
\end{align*}
The distributed implementation meets the prescribed performance,
cf. Figure~\ref{fig:Lyap} and is free of Zeno behavior, as guaranteed
by Theorem~\ref{thm:dist}. This implementation triggers less often
than the centralized derivative-based approach and, as expected, more
often than the centralized \pbb design, cf.  Table~\ref{tab:results}.



\section{Conclusions}
We have developed a novel framework for event-triggered control design
that meets a prescribed performance regarding convergence.  The
proposed approach allows for greater flexibility in prescribing update
times by allowing the certificate to gradually deviate from strictly
decreasing in proportion to the performance residual.  We have shown
analytically how, for exponential performance specifications, the
resulting trigger design exhibits an improved MIET with respect to the
derivative-based approach.  We have taken advantage of the flexibility
of the proposed approach to design intrisically Zeno-free triggers for
network systems that rely on distributed computation and communication
and are applicable for a general class of systems.  Future work will
seek to generalize the guarantees on an improved MIET with respect to
the derivative-based approach and the distributed trigger design for
network systems beyond exponential performance specifications.  We
also plan to explore the extension of the \pbb trigger
design framework to deal with Zeno-free output feedback stabilization,
handle actuation delays, and cope with scenarios where triggers cannot
be evaluated continuously.


\appendix
\begin{proof}[Proof of Lemma~\ref{lem:trackbound}]
  We begin the proof by writing the dynamics of the tracking error,
  \begin{align*}
    \dot \epsilon &= \dot y - \ones\ones^\top \dot W /N
    \\
    &= \dot W - \rho \mathbf{L} (\epsilon-\ones\ones^\top W/N) -
    \ones\ones^\top \dot W/n
    \\
    &= -\rho \mathbf{L} \epsilon + (\mathbf{I}-\ones\ones^\top/N)\dot
    W
  \end{align*}
  where we have used the fact that $\mathbf{L}\ones =0$. Note also
  that $\ones^\top \dot \epsilon = 0$, so $\ones^\top \epsilon = 0$ by
  construction. Hence, at all time, there is no component of
  $\epsilon$ along the eigenvector $\ones$ associated with the
  eigenvalue $0$ of the Laplacian matrix $\mathbf{L}$. Consequently,
  we can bound
  \begin{align*}
    \frac{d}{dt} \|\epsilon\|^2 &= -\rho\epsilon^\top
    (\mathbf{L}+\mathbf{L}^\top) \epsilon +2\epsilon ^\top
    (\mathbf{I}-\ones\ones^\top/N)\dot W
    \\
    &\leq -2\rho \lambda_2 \|\epsilon\|^2 +
    2\|\epsilon\|\|\mathbf{I}-\ones\ones^\top/N\|\|\dot W\|
    \\
    &\leq -2\rho \lambda_2 \|\epsilon\|^2 + 2c_{\dot
      W}\exp(-rt)\|\epsilon\|
  \end{align*}
  for time $t\in[t_k,t_{k+1})$.  It can be verified through
  substitution that the solution
  \begin{multline*}
    v = \frac{c_{\dot W}}{\rho\lambda_2-r} \exp(-rt)+\left(v(0) -
      \frac{c_{\dot W}}{\rho\lambda_2-r}\right)\exp(-\rho \lambda_2
    t).
  \end{multline*}
  satisfies the Bernoulli differential equation~\cite{ELI:56}
  $$
  v\frac{dv}{dt} = -\rho \lambda_2 v^2 + c_{\dot W} \exp(-rt)v.
  $$ 
  Note that when $v\neq 0$, this reduces to
  $$
  \frac{dv}{dt} = -\rho \lambda_2 v + c_{\dot W} \exp(-rt),
  $$ 
  which is linear, with the right-hand side locally Lipschitz in~$v$.
  Then, with $v(0)=\|\epsilon(0)\|\neq 0$, we can deduce
  $\|\epsilon\|^2 \leq v^2$ by applying the Comparison
  Lemma~\cite[Lemma 3.4]{HKK:02}.  Whenever $\|\epsilon\|=0$, it is
  possible (depending on $\dot W$) for $\epsilon$ to remain zero for
  some time interval. On such interval, the Comparison Lemma does not
  apply; however, the case is trivial, and the bound $\|e\|^2 \leq
  v^2$ still holds.  Finally, by noting that $v \geq 0$ because
  $v(0)\geq 0$, we obtain $\|\epsilon\| \leq \vert v \vert = v$ as
  stated.
  %
\end{proof}

\begin{proof}[Proof of Lemma~\ref{lem:signal-exp}]
  Note that, since $F$ and $\kappa$ are Lipschitz, then $f$ is
  Lipschitz too.  Consider the column vector composed of
  $\{V_i\}_{i=1}^N$ and let $J_V(x)$ be its Jacobian.  Then, because
  each $V_i$ have Lipschitz gradients, there exist constants $L_{dV}$
  and $L_f$ on the compact sublevel set $\setdefb{x}{V(x)\leq V(x_0)}$
  such that
  \begin{align}\label{eq:dotWxe-bound}
    \|\dot W^{xe}\| &= \|\big((\sigma-1)c_\alpha x^\top -c_\gamma
    e^\top + (r+c_\beta) J_V(x)\big)f(x,e)\| \notag
    \\
    &\leq  \big((1-\sigma)c_\alpha \|x\| +c_\gamma \|e\| + (r+c_\beta)
    L_{dV}\|x\|\big) \notag
    \\
    &\qquad \times L_f(\|x\|+\|e\|) .
  \end{align}
  We next bound the quadratic terms $\|x\|^2$, $\|e\|^2$ and
  $\|x\|\|e\|$ in terms of $V(x_k)\exp(-r\Delta t_k)$ for the duration
  of the interval $[t_k,t_k +\tau^\deriv_\sigma)$. First, knowing that
  $V(x)\leq V(x_k)\exp(-r\Delta t_k)$ over the interval, we can
  immediately bound $\|x\|^2\leq V(x_k)\exp(-r\Delta t_k)/c_1$. Next,
  for $\|e\|^2$, recall that $\tau^\deriv_\sigma$ is the minimum
  inter-event time for the derivative-based design, and we can
  therefore bound
  \begin{align*}
    \|e\|^2 &\leq
    (1/c_\gamma)\big((1-\sigma)c_\alpha\|x\|^2-rV(x)\big)
    \\
    &\leq (1/c_\gamma)((1-\sigma)c_\alpha/c_1-r)V(x) 
    \\
    &\leq (1/c_\gamma)((1-\sigma)c_\alpha/c_1-r)V(x_k)\exp(-r\Delta
    t_k) ,
  \end{align*}
  for $t \in [t_k,t_k +\tau^\deriv_\sigma)$. Finally, it follows that
  \begin{align*}
    \|x\|\|e\| \leq
    \sqrt{\frac{(1-\sigma)c_\alpha/c_1-r}{c_1c_\gamma}}V(x_k)\exp(-r\Delta
    t_k), 
  \end{align*}
  for $t \in [t_k,t_k +\tau^\deriv_\sigma)$. Substituting the bounds
  back into~\eqref{eq:dotWxe-bound} leads to the identification
  of~$\Omega^{xe}>0$, proving the claim for~$\|\dot W^{xe}\|$.
  
  For the bound of $\|\dot W^{x}\|$, we consider the entire time
  interval $t \in [t_k,t_{k+1})$.  Using the performance satisfaction,
  we bound
  \begin{align*}
    \|x\|^2\leq V(x)/c_1\leq V(x_0)\exp(-rt)/c_1 .
  \end{align*}
  From the trigger condition and $c_\beta > (1-\sigma)(c_\alpha/c_1)-
  r$,
  \begin{align*}
    c_\gamma \|e\|^2 &\leq c_\beta
    V(x_0)\exp(-rt)-(r+c_\beta+(\sigma-1)\frac{c_\alpha}{c_1})V(x)
    \\
    &\leq c_\beta V(x_0)\exp(-rt)
  \end{align*}
  The result now follows using the same line of reasoning as in the
  proof of the bound for $\| \dot W^{xe}\|$ to conclude the existence
  of~$\Omega^{x}>0$ as stated.
\end{proof}

For the sake of completeness, we state the following result on the
sample-and-hold error bound. 

\begin{lemma}\longthmtitle{Sample-and-Hold Error Bound~\cite[Thm
    III.1]{PT:07}}\label{lem:errBnd}
  Consider the sample-and-hold nonlinear system~\eqref{sys:NL}. If the
  functions $f$ is Lipschitz with a constant $L_f$, then for $t \in
  [t_k,t_k+1/L_f)$, the state deviation is bounded as
  $$
  \|e\| \leq \phi(t-t_k)\|x\|
  $$
  where $\phi(\tau)=\frac{L_f \tau}{1-L_f \tau}$.~\hfill$\blacksquare$
\end{lemma}

\bibliographystyle{ieeetr}

\vspace*{-7ex}

\begin{IEEEbiography}[{\includegraphics[width=1in,height=1.25in,clip,keepaspectratio]{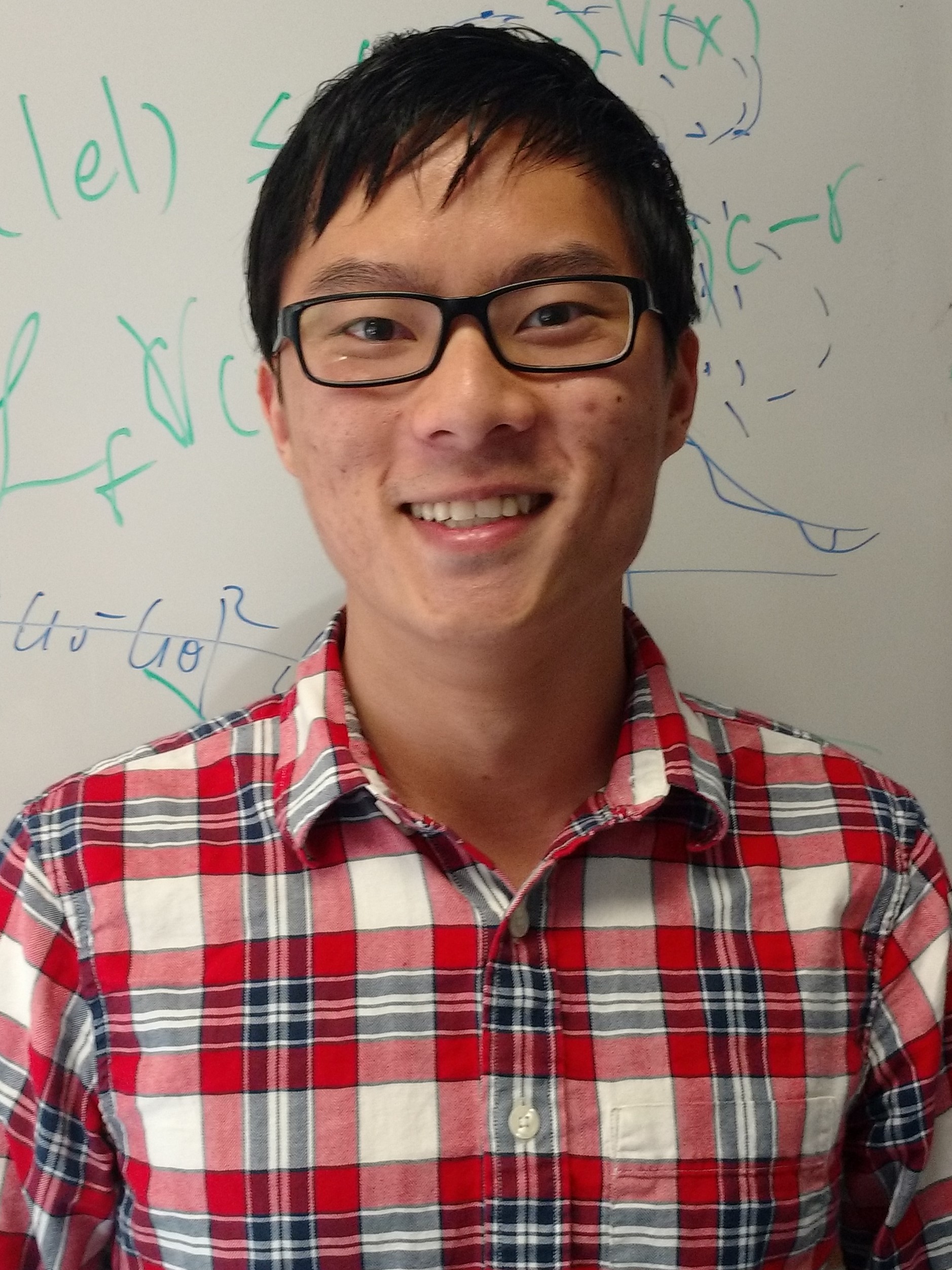}}]{Pio
    Ong}
  received a Bachelor's degree in Aerospace Engineering from
  University of California, San Diego in 2012, and a Master's degree
  in Astronautical Engineering from University of Southern California
  in 2013. During Fall and Winter of 2014, he worked at Space
  Exploration Technologies Corp (SpaceX).  Currently, he is a
  Ph.D. student in the Department of Mechanical and Aerospace
  Engineering at the University of California, San Diego, working as a
  research assistant and occasionally as a teaching assistant under
  the supervision of Professor Jorge Cort\'{e}s. His current research
  interests include multiobjective optimization, event-triggered
  control, human-robot interaction, and safety critical control.
\end{IEEEbiography}

\vspace*{-7ex}

\begin{IEEEbiography}[{\includegraphics[width=1in,height=1.25in,clip,keepaspectratio]{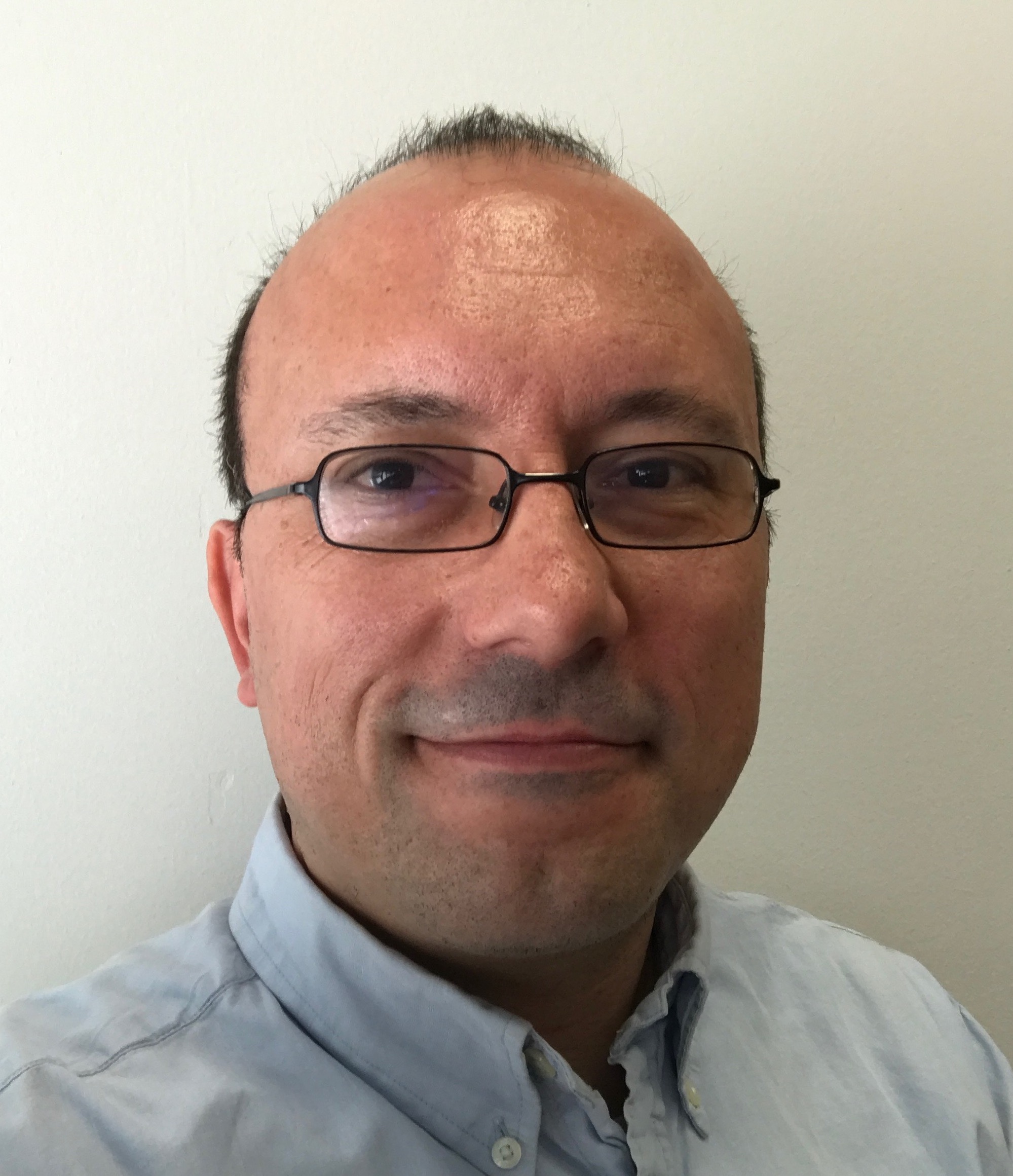}}]{Jorge
    Cort\'{e}s}
  (M'02, SM'06, F'14) received the Licenciatura degree in mathematics
  from Universidad de Zaragoza, Zaragoza, Spain, in 1997, and the
  Ph.D. degree in engineering mathematics from Universidad Carlos III
  de Madrid, Madrid, Spain, in 2001. He held postdoctoral positions
  with the University of Twente, The Netherlands, and the University
  of Illinois at Urbana-Champaign, USA. He was an Assistant Professor
  with the Department of Applied Mathematics and Statistics,
  University of California, Santa Cruz, USA, from 2004 to 2007. He is
  currently a Professor in the Department of Mechanical and Aerospace
  Engineering, University of California, San Diego, USA.  
  He is a Fellow of IEEE
  and SIAM.
  His current research interests include distributed control and
  optimization, network science, nonsmooth analysis, reasoning and
  decision making under uncertainty, network neuroscience, and
  multi-agent coordination in robotic, power, and transportation
  networks.
\end{IEEEbiography}

\end{document}